\documentclass{amsart}
\pdfoutput=1
\usepackage[headings]{fullpage}
\usepackage{amssymb,epic,eepic,epsfig,amsbsy,amsmath,amscd,color,amsthm}
\usepackage[all,cmtip]{xy}
\usepackage{amssymb,amsmath,amscd,color}
\usepackage{graphicx}
\usepackage{texdraw}
\usepackage{url}
\usepackage{mathrsfs}

\usepackage[utf8]{inputenc}
\usepackage{fontenc}
\usepackage{amsfonts}

\usepackage{tikz}
\tikzstyle cross=[preaction={draw=white, -, line width=6pt}]
\tikzstyle normal=[thick]
\usepackage{braids}
\usetikzlibrary{arrows,decorations.markings}

\usepackage{relsize,exscale}

\usetikzlibrary{arrows,decorations.markings}
\usetikzlibrary{cd}

\usepackage{calc}
                        \theoremstyle{plain}

\newtheorem{theorem}{Theorem}[section]

\newtheorem{thm}{Theorem}
\newtheorem{lemma}[theorem]{Lemma}

\newtheorem{coro}[theorem]{Corollary}
\newtheorem{proposition}[theorem]{Proposition}
\newtheorem{prop}[theorem]{Proposition}
\newtheorem{fact}[theorem]{Fact}

\newtheorem*{question}{Question}
\newtheorem*{oquestion}{Open Question}

\newtheorem{defn}[theorem]{Definition}

\theoremstyle{definition}

\newtheorem{rmk}[theorem]{Remark}
\newtheorem{example}[theorem]{Example}
\newtheorem*{Not}{Notations}

\theoremstyle{definition}

\def\BC{\mathbb C}
\def\BN{\mathbb N}
\def\BZ{\mathbb Z}
\def\BR{\mathbb R}
\def\BI{\mathbb I}

\def\BV{\mathbb V}

\def\CC{\mathcal C}

\def\CG{\mathcal G}
\def\CH{\mathcal H}
\def\CI{\mathcal I}

\def\CT{\mathcal T}

\def\Id{\mathrm{Id}}
\def\fS{\mathfrak S}

\def\cV{\mathcal V}

\def\be { \begin{equation} }
\def\ee { \end{equation} }

\def\bpm { \begin{pmatrix} }
\def\epm { \end{pmatrix} }

\def\B{\mathcal B}

\newcommand{\ro}{r}

\newcommand{\Hr}{H_r}

\newcommand{\V}{\mathsf{V}}

\newcommand{\slt}{{\mathfrak{sl}(2)}}
\newcommand{\Uq}{{U_q\slt}}

\newcommand{\UsltH}{{U_q^{H}\slt}}
\newcommand{\Ubar}{{\wb U_q^{H}\slt}}

\newcommand{\cat}{\mathscr{C}}

\newcommand{\qn}[1]{{\left\{#1\right\}}}

\newcommand{\qd}{{\mathsf d}}

\newcommand{\End}{\operatorname{End}}
\newcommand{\Hom}{\operatorname{Hom}}

\newcommand{\C}{\ensuremath{\mathbb{C}} }
\newcommand{\Z}{\ensuremath{\mathbb{Z}} }

\newcommand{\wb}{\overline}

\newcommand{\RR}{\operatorname{R}}
\newcommand{\RRR}{\mathcal{R}}
\newcommand{\perm}{\operatorname{perm}}

\newcommand{\Cob}{\mathcal{C}ob}

\newcommand{\Bn}{\mathcal{B}_n}
\newcommand{\PBn}{\mathcal{PB}_n}
\newcommand{\PB}{\mathcal{PB}}

\newcommand{\Sk}{\mathfrak{S}}

\newcommand{\Homeop}{\mathop{Homeo}^{+}}
\newcommand{\Homeo}{\mathop{Homeo}}
\newcommand{\Mod}{\mathop{Mod}}

\newcommand{\Mcg}{\mathop{M}}
\newcommand{\PMcg}{\mathop{PM}}
\newcommand{\Caping}{\mathop{Cap}}

\newcommand{\piz}{\mathop{\pi}_0}

\newcommand{\Vect}{\mathop{Span}}
\newcommand{\SL}{\mathop{SL}}

\newcommand{\PGL}{\mathop{PGL}}
\newcommand{\PSL}{\mathop{PSL}}

\newcommand{\Matrice}{\mathop{Mat}}

\newcommand{\RT}{\mathop{\mathcal{RT}}}

\newcommand{\qbin}[2]{\left[\begin{array}{c}
      #1 \\
      #2 \end{array}\right]}
\newcommand{\sjtop}[6]{\left|\begin{array}{ccc}#1 & #2 & #3 \\#4 & #5 &
      #6\end{array}\right|}

\def\cB{\mathcal B}

\newcommand{\bapp}{\left. \begin{array}{rcl}}
\newcommand{\eapp}{\end{array} \right.}
\newcommand{\bfct}{\left\lbrace \begin{array}{rcl}}
\newcommand{\efct}{\end{array} \right.}

\title{Non semi-simple TQFT of the sphere with $4$ punctures}
\author{Jules Martel}
\date{}

\begin{document}

\maketitle

\begin{abstract}
In this work, we compute the representation of the mapping class group of the sphere with $4$ punctures arising from the non semi-simple TQFT \cite{BCGP2}. We show that it is faithful. Lastly, we compare quantum representations of punctured spheres in general with those of braid groups. 
\end{abstract}

\tableofcontents

\section{Introduction}

In this introduction, we first define the mapping class groups. Then we provide a short overview of TQFTs leading to the non semi-simple case. Thank to these definitions, we lastly introduce the content of this paper. 

\subsection{Mapping class groups}% and punctured sphere}

Let $S$ be an oriented surface, we denote by $\Homeop(S,\partial S)$ the group of orientation-preserving homeomophisms that fix the boundary pointwise, endowed with compact-open topology, and let $\Homeo_0(S,\partial S)$ be the connected component of the identity.

\begin{defn}[Mapping Class Group]
The {\em Mapping Class Group}  of $S$ is the group of isotopy classes of orientation preserving homeomorphisms of $S$. Namely:
\[
\begin{array}{rcl}
\Mod (S) & = & \piz(\Homeop(S,\partial S)) \\
& = & \Homeop (S,\partial S) / \Homeo _0 (S,\partial S)
\end{array} .
\]
\end{defn}

\begin{defn}[Braid Groups]
The {\em braid group} $\Bn$ is the mapping class group of the disk with $n\in \BN$ punctures. It has a presentation with $n-1$ generators $\sigma_1,\ldots,\sigma_{n-1}$ satisfying the so called {\em ``braid relations"}:
\begin{align*} \sigma_i \sigma_j = \sigma_j \sigma_i & \text{ if } |i-j| \le 2  \text{ and, } \\
\sigma_i \sigma_{i+1} \sigma_i = \sigma_{i+1} \sigma_i \sigma_{i+1} & \text{ for } i=1,\ldots, n-2.
\end{align*}
The {\em pure braid group} $\PB_n$ is the subgroup of $\Bn$ consisting in mapping classes arising from homeomorphisms fixing the punctures pointwise. 
\end{defn}

\begin{defn}
We denote the mapping class group of the sphere with $n \in \BN$ punctures by $\Mcg (0,n)$.
\end{defn}

%
%In this definition it is equivalent to consider isotopy instead of homotopy or diffeomorphisms instead of homeomorphisms, it would result to the same group.

\begin{example}[$2$ and $3$-punctured sphere]
The mapping class group of the sphere with $3$ points removed, denoted $\Mcg (0,3)$, is exactly $\fS _3$, the group of permutation of three elements. In the case of two punctures one has $\Mcg (0,2) = \BZ / 2 \BZ$.
%\\
%
%To see this we must first remark that all the arcs having two of the punctures as endpoints are isotopic in the $3$-punctured sphere.
%
%Then if we consider a mapping class $\Phi$ in the kernel of $\Mod (\Gbgn) \to \fS_3$ that fixes the marked points, it sends an arc $\alpha$ with endpoints $p_1,p_2$ to an arc with the same endpoints so that $\Phi (\alpha)$ and $\alpha$ are isotopic. It follows that $\Phi$ is isotopic to a map fixing $\alpha$ (see \cite{F-M}, Proposition 1.11). By cutting along $\alpha$ what we get is a once punctured disk, so the mapping class is trivial following the previous example. Then the morphism onto $\fS_3$ is actually an isomorphism.\\
%
%Using same arguments, one can prove that $\Mcg (0,2) = \BZ / 2 \BZ$ .
\end{example}

The first non trivial example for closed surface is the torus, it is a main example as it inspires the general classification of mapping classes.

\begin{example}[Torus]
Let $T^2$ be the torus (genus $1$ closed surface). There is a surjective homomorphism:
\[
\Mod(T^2) \to SL(2,\BZ)
\]
given by the action of mapping classes on $H_1(T^2,\BZ)$. By theorem it gives an isomorphism between $\Mod(T^2)$ and $\PSL(2,\BZ)$, see \cite[Theorem 2.5]{F-M}.
\end{example}

\subsection{Non semi-simple TQFTs}\label{recallsonTQFT}

In \cite{BCGP2}, Blanchet--Costantino--Geer--Patureau construct a TQFT from the non semi-simple category $\cat$ of $\Ubar$ weight modules. We present here a non exhaustive summary of the construction, recalling first what is a TQFT and what informations are provided by these theories. This part gives ideas before fixing concrete notations for the precise case of interest (which will be done rigorously in Section \ref{sectionm(04)faithful}).

\subsubsection{Topological quantum field theory}

\begin{defn}[Category of cobordisms]
An oriented $(n + 1)$-manifold $M$ with boundary decomposed as $\partial M = - \Sigma_1 \bigsqcup \Sigma_2$, where $\Sigma_1$, $\Sigma_2$ are oriented $n$-manifolds, and $-\Sigma_1$ means $\Sigma_1$ with reversed orientation, is called a cobordism from $\Sigma_1$ to $\Sigma_2$. Given a cobordism $M_1$, from $\Sigma_1$ to $\Sigma$, and a cobordism $M_2$, from $\Sigma$ to $\Sigma_2$, one can glue these together along $\Sigma$ to obtain a cobordism from $\Sigma_1$ to $\Sigma_2$. Let the category $\Cob_{n+1}$ be the one whose objects are the oriented $n$-manifolds, whose morphisms are equivalence classes of cobordisms, and where gluing plays the role of composition. Two cobordisms from $\Sigma_1$ to $\Sigma_2$ are called equivalent if they are isomorphic rel. boundary (i.e. the isomorphism is required to be the identity on $\Sigma_1$ and $\Sigma_2$). Taking equivalence classes ensures that composition is associative, and the product manifold $\left[0,1  \right] \times \Sigma$ plays the role of the identity morphism of $\Sigma$. Observe that this category has an involution (given by orientation reversal) and a monoidality structure (given by disjoint union). 
\end{defn}

\begin{rmk}
This is the basic definition of the category of cobordisms. We usually restrict it to compact surfaces and allow richer cobordism. {\em ``Extra decorations"} of cobordisms can be of the following types: cobordism containing a banded link, decorated points in the surface, or cohomology class associated with objects. Then an appropriate generalization of the notion of isomorphism of cobordism is required. 
\end{rmk}

\begin{defn}[$(2+1)$-TQFT]
Let $\Cob_{2+1}$ be the category of $3$-dimensional cobordisms, and $k$ be a commutative ring. A {\em $(2+1)$-TQFT} is a functor:
\[
V: \Cob_{2+1} \to k-\text{modules}
\]
satisfying:
\begin{itemize}
\item[(Monoidality)] $V(\Sigma_1 \bigsqcup \Sigma_2) = V(\Sigma_1) \otimes V(\Sigma_2)$.
\item[(Duality)] $V(-\Sigma_1) = V(\Sigma_1)^*$ where $*$ stands for the dual module.
\item[(Unit)] $V(\emptyset)= k$.
\end{itemize}
\end{defn}

Again this is the initial definition of a TQFT, while extra decorated cobordisms need an adapted definition of a TQFT functor. 

\begin{rmk}[Mapping class group representations]\label{mcgTQFTrep}
Let $\Sigma$ be a surface, $\Phi$ a diffeomorphism of it and $V$ a $(2+1)$-TQFT. The mapping cylinder of $\Phi$ is the following manifold:
\[
I_{\Phi} = \left(\Sigma \times \left[ 0,1 \right] \right) 
\]
with $(x,1) \sim x$ and $(x,0) \sim \Phi(x)$. The manifold $I_{\Phi}$ is a cobordism between $\Sigma$ and itself, so that $V \left( I_{\Phi} \right)$ is an endomorphism of $V\left( \Sigma \right)$. The functoriality of $V$ together with the notion of isomorphism of cobordisms imply that:
\[
\bapp
\Mod \left( \Sigma \right) & \to & \End\left(V\left( \Sigma \right) \right) \\
\Phi & \mapsto & V( \Phi) := V \left( I_{\Phi} \right)  
\eapp
\]
is a representation of $\Mod \left( \Sigma \right)$ over $V\left( \Sigma \right)$. This remark shows that a TQFT functor provides a representation of the mapping class group for all surfaces. 
\end{rmk}

\subsubsection{Category of tangles and Reshetikhin -- Turaev functor}

\begin{defn}[Tangles]\label{tangles}
A {\em tangle} with $k$ inputs and $l$ outputs is a finite system of disjoint smoothly embedded oriented arcs and circles in $\BR^2 \times \left[0 , 1 \right]$ such that the endpoints of the arcs are the points $(1,0,0) , \ldots , (k,0,0)$ and $(1,0,1) , \ldots , (l,0,1)$. The circles lie in $\BR^2 \times \left(0 , 1 \right)$.

A tangle is {\em framed} if it is equipped with a non-singular normal vector field equal in the endpoints of the arcs to the vector $(0,-1,0)$. 

Two (framed) $(k,l)$-tangles $L_1$ and $L_2$ are said to be {\em isotopic} if $L_1$ may be smoothly deformed into $L_2$ staying in the class of (framed) $(k,l)$-tangles during the deformation. 
\end{defn}

\begin{rmk}
Braids are special types of tangles. Namely: a braid of $\Bn$ is a $(n,n)$-tangle with no circle. 
\end{rmk}

\begin{defn}[Category of Tangles]\label{catoftangles}
The  category of tangles $\CT$ is the category whose objects are non negative integers, and a morphism from $k$ to $l$ is a $(k,l)$-tangle. Let $f: k \to l$ and $g : l \to m$, the morphism $fg$ is represented by the tangle obtained by attaching $f$ on the top of $g$. 
\end{defn}

\begin{defn}[Category of colored tangles]\label{catofcoloredtangles}
Let $C$ be a category. A $C$-colored tangle is a tangle where every component is equipped with an object of $C$. More precisely, the category of $C$-colored tangles $\CT_C$ is the category whose objects are finite sequences $((V_1 , \epsilon_1) , \ldots , (V_m, \epsilon_m))$, where $V_1 , \ldots , V_m$ are objects of $C$ and $\epsilon_1 , \ldots , \epsilon_m \in \lbrace + , - \rbrace$. A morphism $\eta \to \eta'$ is an isotopy type of a $C$-colored framed tangle such that $\eta$ (resp. $\eta'$) is the sequence of colors and directions of those tangles which hit the bottom (resp. top) boundary endpoints. (The sign $+$ stands for the downward direction, while the $-$ the upward.)

This category is monoidal. The tensor product of two sequences $\eta$ and $\eta'$ is given by the concatenation of sequences. The tensor product of two morphisms $f$ and $g$ is obtained by placing the colored framed tangle representing $f$ to the left of the one representing $g$. 
\end{defn}

The idea of Reshetikhin and Turaev is to use a category of modules over a quantum algebra that has nice enough properties allowing to build topological invariants of tangles. In this paper we are interested in quantum algebra arising from the Lie algebra $\slt$. Let $\Uq$ be the quantized deformation algebra arising from the enveloping algebra of $\slt$ (this will be precised in what follows). 

\begin{theorem}[Reshetikhin-Turaev functor, \cite{RT}]\label{RTfunctor}
Let $\CC$ be the category of $\Uq$-modules. There exists a monoidal functor $\RT$ between $\CT_{\CC}$ and $\CC$. 
\end{theorem}

The above theorem is loosely stated as one must refine the category of $\Uq$-modules before applying the construction for it to work. We will see cases of $\Uq$-modules categories for which the theorem holds. Historically there are two $\Uq$-modules category (for infinite versions of $\Uq$) for which this theorem holds: the semi-simple theory first introduced in \cite{RT}, and the non semi-simple one introduced for instance in \cite{CGP2}. For a categorical approach to the non semi-simple construction of a Reshetikhin-Turaev functor, see \cite{Marco}, while in this work we are interested in the concrete $\Uq$ case. This functor then also provides braid representations as braids are a sub-category of tangles. 
%
%\begin{prop}\label{braidingfromRT}
%The braid representations coming from the $\Uq$ R-matrix (Proposition \ref{naivebraiding}) are restrictions of the functor $\RT$ to braids. 
%\end{prop}
%
%We will use the general term of {\em quantum representations} or {\em $\Uq$-representations} of the braid group in what follows to designate the representations built from the $\RT$-functor, or equivalently by use of the $R$-matrix. 

\subsubsection{Universal construction}

From certain categories of quantum groups modules, the Reshetikhin -- Turaev functor $\RT$ (see Definition \ref{RTfunctor}) provides invariants of links. The following is well known.

\begin{theorem}[Lickorish -- Wallace and Kirby Theorem]\label{lickorishwallace}
Any closed, orientable, connected $3$-manifold may be obtained by performing {\em Dehn surgery} on a framed link in the $3$-sphere with $\pm 1$ surgery coefficients. Two framed links give the same manifold if and only if they are related by a series of {\em Kirby moves}.
\end{theorem}

We don't give the definitions of Dehn surgery nor Kirby moves, see \cite{BHMV} for instance. We want to emphasize that the $\RT$ functor gives a quantum invariant of framed link that has been generalized in \cite{RT2} to $3$-manifold invariants, applying the above theorem. We call such invariants of manifolds {\em quantum invariants} by extension. In \cite{RT2} the initial category of quantum groups modules is semi-simple. 

In \cite{BHMV}, the authors present a {\em universal construction} of TQFT. Namely they suggest a technique to get a TQFT from a family of quantum invariants of three manifolds, using a natural pairing. The technique works with the Reshetikhin-Turaev invariants from \cite{RT2} and gave rise to the semi-simple Reshetikhin-Turaev TQFT's. In \cite{CGP2} the authors succeed in constructing a quantum invariant from a non semi-simple category of quantum groups modules. An adaptation of the universal construction is performed in \cite{BCGP2} and provides  a {\em non semi-simple TQFT}.

\begin{theorem}[{\cite[Theorem~1.1]{BCGP2}}]\label{BCGPfunctor}
There exists a monoidal functor $\BV : \mathcal{C}ob \to \mathcal{G}r\mathcal{V}ect$ from the category of decorated surfaces and decorated cobordisms to the category of finite dimensional $\BZ$-graded vector spaces. This functor is built from the non semi-simple category $\cat$ of $\Ubar$ weight modules (defined in what follows, Section \ref{sectionm(04)faithful}). The category $\cat$ is used to decorate the cobordisms, see the definition of decorations in \cite[Subsection~3.3]{BCGP2}.

Moreover the mapping class group representations (Remark \ref{mcgTQFTrep}) preserve the grading.
\end{theorem}

The non semi-simplicity of $\cat$ implies richer topological information than in the case of the classical Reshetikhin-Turaev TQFTs (semi simple). For instance, the following theorem about mapping class group representations is a strong improvement compared to the original Reshetikhin-Turaev TQFT's.

\begin{theorem}[{\cite[Theorem~1.3]{BCGP2}}]
The action of a Dehn twist along a non-separating curve of a surface $\Sigma$ has infinite order on $\BV\left( \Sigma  \right)$. 
\end{theorem}

The latter suggests that the new family of representations of mapping class groups provided by the non semi-simple TQFTs is richer and one would be interested in the question of their faithfulness. We investigate a precise case in the following section. We end these recalls by a far from exhaustive presentation of the universal construction providing non semi-simple TQFTs from $\RT$-functor.

\begin{rmk}[One ingredient of the universal construction]\label{sketchofBHMV}
We present loosely the universal construction introduced in \cite{BHMV} and performed in a more sophisticated way in \cite{BCGP2} giving rise to non semi-simple TQFTs. 

Let $\Sigma$ be a surface, $V_1(\Sigma)$ be the complex vector space generated by all cobordisms between $\emptyset$ and $\Sigma$ and $\widehat{V_1}$ the one generated by all cobordisms between $\Sigma$ and $\emptyset$. There is a pairing:
\[
\bapp
V_1(\Sigma) \times \widehat{V_1}(\Sigma) & \to & \BC \\
(C,C') & \mapsto & N_r^0(C \sharp_{\Sigma} C')
\eapp
\]
where $C \sharp_{\Sigma} C'$ is the closed $3$-dim manifold obtained by gluing $C$ and $C'$ along their common boundary, and $N_r^0$ is the quantum invariant of closed manifold constructed in \cite{CGP2}. By making the quotient of $V_1$ by the kernel of this pairing, one obtains the TQFT module associated to $\Sigma$. To deal with details of the pairing and of the quotient in the case of decorated cobordisms, one should follow \cite{BCGP2}. In what follows we will refer to \cite{BCGP2} notations to fit with the $\cat$-decorated formalism. 
\end{rmk}

\subsection{Content of the paper}

In Section \ref{HyperellipticInvol} we present a point of view for $\Mcg(0,4)$ using the mapping class group of the torus known to be isomorphic to $\PSL(2,\BZ)$. We recall basics about this last group and then the construction relating the torus and the punctured sphere: the {\em hyperelliptic involution}. In Section \ref{sectionm(04)faithful} we construct explicitly the non semi-simple TQFT representation of $\Mcg(0,4)$ using a trivalent graph basis. Then, comparing this representation with the one given by the hyperelliptic involution we show the faithfulness of the TQFT representation. In the last Section \ref{TQFTbiggerspheres} we give a quick overview for sphere with more punctures, by relating their TQFT representations with quantum representations of braid groups. 
%
%In this section, we build the non semi-simple TQFT's representations of $\Mcg (0,4)$ (the mapping class group of the four times punctured sphere using a precise basis). Then we state how they contain the hyperelliptic representations of $\Mcg(0,4)$ presented in Example \ref{4sphere}. This leads to the faithfulness of the representation. 

{\bf Acknowledgment} This work was achieved during the PhD of the author that was held in the {\em Institut de Mathématiques de Toulouse}, in {\em Université Paul Sabatier, Toulouse 3}. The author thanks very much his advisor Francesco Costantino for suggesting this problem, and for fruitful remarks that led to this paper.  %The author is also very grateful to Emmanuel Wagner for his valuable comments and advises about this work.

\section{Hyperelliptic involution and $\PSL(2,\BZ)$}\label{HyperellipticInvol}

In this section we present the {\em hyperelliptic involution} of the torus that will be used to study the mapping class group $\Mcg(0,4)$ using that of the torus, namely $\PSL(2,\BZ)$. First we recall some property of $\PSL(2,\BZ)$. 

\subsection{Recalls on representations of $\PSL(2,\BZ)$}
We consider the three following presentations of groups:
\[
G_1=\left\langle a,b \text{ } | \text{ } aba=bab \text{ , } (aba)^4=1 \right\rangle
\]
\[
G_2=\left\langle s,t \text{ } | \text{ } s^2=t^3 \text{ , } t^4=1 \right\rangle
\]
\[
H=\left\langle s,t \text{ } | \text{ } s^2=t^3=1 \right\rangle
\]
and let $f$ be the following morphism:
\[
f: \left\lbrace \begin{array}{rcl}
G_1 & \to & \SL(2,\BZ)\\
a & \mapsto & A \\
b & \mapsto & B
\end{array} \right.
\]
with:
\[
A= \begin{pmatrix} 1 & 1 \\ 0 & 1 \end{pmatrix} \text{ , } B = \begin{pmatrix} 1 & 0 \\ -1 & 1 \end{pmatrix}.
\]
Then $f$ is a homomorphism.
\begin{fact}
\begin{itemize}
\item The groups $G_1$ and $G_2$ are isomorphic, up to the following inverse subsitutions:
\[
s=ab \text { , } t=aba \text{ and } a=s^{-1}t \text{ , } b=t^{-1}s^2
\]
we call it $G$ from now on.
Moreover the group $G$ is isomorphic to the quotient of the braid group $\B_3$ by the central subgroup generated by $(\sigma_1 \sigma_2 \sigma_1)^4$.
\item The group $H$ is isomorphic to the quotient of $G$ by the group generated by $s^3=t^2$, so that it is isomorphic to the quotient of $\B_3$ by its central subgroup $Z(\B_3)$.
\end{itemize}
\end{fact}

We introduce here the following matrices, in order to relate the different presentations to the matrix representation $f$:
\[
S=f(s)=AB= \begin{pmatrix} 0 & 1 \\ -1 & 1 \end{pmatrix} \text{ , } T=f(t)=ABA=\begin{pmatrix} 0 & 1 \\ -1 & 0 \end{pmatrix} .
\]
We get $T^2 = - I_2$, so that $f$ provides a morphism $\bar{f}: H \to \PSL(2,\BZ)$.
\begin{prop}
The morphisms $f: G \to \SL(2,\BZ)$ and $\bar{f}= H \to \PSL(2,\BZ)$ are isomorphisms.
\end{prop}

\subsection{Hyperelliptic involution and punctured spheres}

For the general case of the sphere with $n$ punctures, there is the following theorem, giving a general presentation of the mapping class group.

\begin{theorem}\label{presentationSphere}
If $n\ge 2$, then $\Mcg(0,n)$ admits a presentation with generators $\sigma_1, \ldots, \sigma_{n-1}$ together with the following defining relations:
\[
\begin{array}{rclr}
\sigma_i \sigma_j & = & \sigma_j \sigma_i , & |i-j| \ge 2 \\
\sigma_i \sigma_{i+1} \sigma_i & = & \sigma_{i+1} \sigma_i \sigma_{i+1} &\\
(\sigma_1 \sigma_2 \cdots \sigma_{n-1})^n & = & 1 &\\
\sigma_1 \cdots \sigma_{n-2} \sigma_{n-1}^2 \sigma_{n-2} \cdots \sigma_1 & = & 1 &
\end{array}
\]
\end{theorem}
In the above theorem, $\sigma_i$ corresponds to the half Dehn twist along some arc relating $p_i$ and $p_{i+1}$, for $i=1, \ldots, n-1$. We will sometimes refer to these generators as {\em braid generators}, as they satisfy the {\em braid relations}. We define the half Dehn twists involved in the above theorem.

\begin{defn}[Half Dehn twist]\label{halfdehntwist}
Let $\alpha$ be an arc in a surface $M$ having endpoints in a subset $Q \subset M$. By half Dehn twist along $\alpha$ we mean the homeomorphism:
\[
\tau_{\alpha}: (M,Q) \to (M,Q)
\]
which is obtained as the result of the isotopy of the identity map $\Id : M \to M$ rotating $\alpha$ in $M$ about its midpoint by the angle $\pi$ in the direction provided by the orientation of $M$. The half-twist $\tau_{\alpha}$ is the identity outside a small neighborhood of $\alpha$ in $M$.
\end{defn}

In the next example we present the hyperellyptic involution representation of the mapping class group of the sphere with four punctures. This will be used extensively to show the faithfulness of the non semi-simple TQFT representation later on. 

\begin{example}[Sphere with $4$ punctures]\label{4sphere}

This example is treated in \cite[Section 2.2.5]{F-M}, but we mainly follow \cite{AMU} which provides matrices for generators that will be useful to study the representations.

The idea is to study $\Mcg(0,4)$ from the torus. Think of the torus as the square with opposite faces identified, say that the left lower vertex is in $\bpm 0 \\ 0 \epm$. Let $\iota$ be the $\pi$-rotation fixing the square, and having four fixed points ($\bpm 0 \\ 0 \epm$,$\bpm 1/2 \\ 0 \epm$,$\bpm 0 \\ 1/2 \epm$,$\bpm 1/2 \\ 1/2 \epm$), we call it the {\em hyperelliptic involution}. The quotient of the torus by the action of $\iota$ is an orbifold that is topologically a sphere with $4$ ramified points. We will use this sphere with this four marked points allowing relations with the well known mapping class group of the torus.

Let $A \in SL(2,\BZ)$ be a matrix, $v \in \left( \frac{1}{2} \BZ \right)^2$ be a vector, and $\phi_{A,v}$ be the transformation of $\BR ^2$:
\[
x \mapsto Ax+v .
\]

This defines a diffeomorphism of $T^2 = \BR ^2 / \BZ ^2$ commuting with $\iota = -\Id$. We keep the notation $\phi_{A,v}$ to designate the diffeomorphism of the quotient $T^2 / \{\pm \Id \}$, the $4$-punctured sphere. By Theorem 3.1 of \cite{AMU}, this association is surjective. More precisely, all the braid generators of $\Mcg(0,4)$, (from Theorem \ref{presentationSphere}), are reached by diffeomorphisms of the form $\phi_{A,v}$. The study of the kernel is given by the exact sequence of \cite[Corollary 3.3]{AMU}:
\[
1 \to N \to \Mcg(0,4) \to PSL(2,\BZ) \to 1
\]

where a mapping class coming from some $\phi_{A,v}$ is sent to the matrix $\pm A \in PSL(2,\BZ)$, and $N=\BZ / 2 \BZ \times \BZ / 2 \BZ$.

Moreover, from \cite[Proposition 2.7]{F-M}, the sequence splits so that $\Mcg(0,4)$ is the semi-direct product $PSL(2,\BZ) \ltimes N$.

The latter is done in a more explicit manner in \cite{Bir} using the general presentation of the mapping class groups of punctured spheres. Namely, from Theorem \ref{presentationSphere}, we get the following presentation for $\Mcg(0,4)$. Let $\sigma_1,\sigma_2,\sigma_3$ be the three generators together with the following relations:
\begin{align}\label{relationM(0,4)}
\sigma_1 \sigma_3 & = \sigma_3 \sigma_1 \\
\sigma_1 \sigma_2 \sigma_1 & =  \sigma_2 \sigma_1 \sigma_2\\
\sigma_3 \sigma_2 \sigma_3 & =  \sigma_2 \sigma_3 \sigma_2\\
(\sigma_1 \sigma_2 \sigma_3)^4 & =  1\\
\sigma_1 \sigma_2 \sigma_3^2 \sigma_2 \sigma_1 & =  1
\end{align}

Let $G$ be the subgroup of $\Mcg(0,4)$ generated by $\sigma_1$ and $\sigma_2$, and let $N$ be the subgroup generated by $a= \sigma_1 \sigma_3^{-1}$ and $b= \sigma_2 \sigma_1 \sigma_3^{-1} \sigma_2^{-1}$.
\begin{lemma}[{\cite[Lemma 5.4.1]{Bir}}]\label{M04PSL2}
The group $M(0,4)$ is the semi direct product of the normal subgroup $N$ and the subgroup $G$. 
\end{lemma}
Then one gets that $G$ is isomorphic to $\PSL(2,\BZ)$ under the following association:
\[
\sigma_1 \leftrightarrow A = \begin{pmatrix} 1 & 1 \\ 0 & 1 \end{pmatrix} \text{ , } \sigma_2 \leftrightarrow B = \begin{pmatrix} 1 & 0 \\ -1 & 1 \end{pmatrix} .
\]

It is easy to check that elements $a$ and $b$ commute and are of order $2$ by simple applications of Relations \ref{relationM(0,4)}. This gives that $N$ is isomorphic to $\BZ / 2 \BZ \times \BZ / 2 \BZ$
\end{example}

\section{TQFT-representations of $\Mcg(0,4)$.}\label{sectionm(04)faithful}

To construct the non semi-simple TQFT we need, as a ground tool, the algebra $\Ubar$ and some modules of it. These objects are first introduce in \ref{Ubar}. Then we construct the representation of $\Mcg(0,4)$ in \ref{Construction} using an explicit basis. Lastly, in \ref{faithfulness} we show that the obtained representation is faithful.

\tikzset{fleche/.style args={#1:#2:#3}{ postaction = decorate,decoration={name=markings,mark=at position #1 with {\arrow[#2,scale=2]{>}}},node[midway,above]{#3}]}}

\tikzstyle directed=[postaction={decorate,decoration={markings,
    mark=at position .65 with {\arrow{latex}}}}]

\tikzset{repmidarrow/.style={postaction=decorate,decoration={markings,mark={between positions 0.5 and 1 step #1 with {\arrow{latex}}}}}}

\subsection{$\Ubar$-modules}\label{Ubar}

Here is a first definition of the quantum enveloping algebra of the Lie algebra $\slt$. 

\begin{defn}[$\Uq$]\label{Uqnaif}
Let $q$ be a complex parameter. We define $U_q= U_q(\slt)$ as the $\BC$-algebra generated by the four generators $E,F,K,K^{-1}$ together with relations:
\[
KK^{-1}=K^{-1}K=1
\]
\[
KEK^{-1}=q^2E,\text{ } KFK^{-1}=q^{-2}F
\]
\[
\left[E, F \right] = \frac{K-K^{-1}}{q-q^{-1}} .
\]\label{relationUq}

$U_q$ is endowed with a coalgebra structure, with coproduct $\Delta$ and counit $\epsilon$ defined as follows:
\[
\begin{array}{rl}
\Delta(E)= 1\otimes E+ E\otimes K, & \Delta(F)= K^{-1}\otimes F+ F\otimes 1 \\
\Delta(K) = K \otimes K, & \Delta(K^{-1}) = K^{-1}\otimes K^{-1} \\
\epsilon(E) = \epsilon(F) = 0, & \epsilon(K) = \epsilon(K^{-1}) = 1
\end{array},
\]\label{relationDeltaUq}
and with an antipode defined as follows:
\[
S(E) = EK^{-1}, S(F)=-KF,S(K)=K^{-1},S(K^{-1}) = K.
\]
This turns $U_q$ into a {\em Hopf algebra}. 
\end{defn}

We present now a slightly modified version of the quantum enveloping algebra of $\slt$, that is presented in large details in \cite{CGP1} for instance. From now on, let $q$ be a root of unity of pair degree: i.e. such that $q^{2r}=1$ for some integer $r\ge 2$.

%\subsubsection{The algebra $\UsltH$}

Let $\UsltH$ be the $\C$-algebra $U_q$ of Definition \ref{Uqnaif} improved with one more generator $H$, so given by generators $E, F, K, K^{-1}, H$ and the relations from $U_q$ together with relations:
\begin{align*}
  HK&=KH,
& [H,E]&=2E, & [H,F]&=-2F.
\end{align*}
The Hopf algebra structure of $\Ubar$ comes from the one of $U_q$ extended by:
\begin{align*}
  \Delta(H)&=H\otimes 1 + 1 \otimes H,
  & \epsilon(H)&=0,
  &S(H)&=-H.
\end{align*}

\begin{defn}[$\Ubar$]\label{algebraUbar}
$\Ubar$ is the Hopf algebra $\UsltH$ modulo the relations $E^\ro=F^\ro=0$. 
\end{defn} 

Let $V$ be a finite dimensional $\Ubar$-module. An eigenvalue $\lambda\in \C$ of the action $H:V\to V$ is called a \emph{weight} of $V$ and the associated eigenspace is called a \emph{weight space}. We call $V$ a \emph{weight module} if $V$ splits as a direct sum of weight spaces and if $K$ acts as the exponential of $H$ on $V$, namely $Kv=q^\lambda v$ if $v$ is a vector of weight $\lambda$.

\begin{defn}[$\Ubar$ braiding, {\cite[Subsection~2.2]{CGP1}}]\label{Ubarbraiding}
Let $\cat$ be the category of finite dimensional weight $\Ubar$-modules, and let $V$ and $W$ be two elements of this category. Let $\RR$ be the R-matrix defined by Ohtsuki in \cite{Oht} with the expression that can be found in \cite[Equation~(5)]{CGP1}. It is not an element of $\Ubar \otimes \Ubar$ so it is not a universal R-matrix, but it yields an operator on $V\otimes W$ as follows:
\[
\RR = q^{H\otimes H} \sum_{n=0}^{r-1} \frac{\{1\}^{2n}}{\{n\}!} q^{n(n-1)/2}E^n \otimes F^n,
\]
where the action $q^{H\otimes H}$ is described for $v$ and $v'$, two weight vectors of weights $\lambda$ and $\lambda'$, respectively as follows:
\[
q^{H\otimes H} ( v\otimes v') = q^{\lambda \lambda'} v\otimes v'.
\]
This way, $\RR$ is a well defined linear map, and gives rise to a braiding:
\[
c_{V,W}: \left. \begin{array}{ccc}
V\otimes W & \to & W \otimes V \\
v\otimes w & \mapsto & \tau ( \RR (v\otimes x))
\end{array} \right.
\]
where $\tau$ is defined by $\tau (v\otimes w) = w \otimes v$.

\end{defn}

The above definition uses quantum numbers defined as follows. 

\begin{defn}
For $n \in \BN$:
\[
\lbrace x \rbrace = q^x - q^{-x} \text{ and } \lbrace n \rbrace ! = \lbrace n \rbrace  \lbrace n-1 \rbrace \cdots \lbrace 1 \rbrace ,
\]
where we use the notation $q^x = e^{\frac{\pi \sqrt{-1} x}{r}}$.
\end{defn}

%\subsubsection{Simple $\Ubar$-modules}\label{simpleUbarmodules}

We focus on a special class of finite dimensional weight modules, the one we will use for the quantum representations construction below.
For each $\lambda\in \C$ there exists a unique $\Ubar$-module $V_\lambda$ which is $r$-dimensional and of highest weight $\lambda + r-1$.  The module $V_\lambda$ has a basis $\{e^{\lambda}_0,\ldots,e^{\lambda}_{r-1}\}$ whose action is given by
%\begin{equation}\label{E:BasisV}
\[
H.e^{\lambda}_i=(\lambda + r-1-2i) e^{\lambda}_i,\quad E.e^{\lambda}_i= \frac{\qn i\qn{i-\lambda}}{\qn1^2}
e^{\lambda}_{i-1} ,\quad F.e^{\lambda}_i=e^{\lambda}_{i+1}.
\]
%\end{equation}

 The module $V_\lambda$ is called \emph{typical} if $\lambda\in  (\C\setminus \Z)\cup r\Z$, \emph{atypical} otherwise.  If $V_\lambda$ is typical then  it is simple, and is generated (as a module) by any of the  basis vectors $e^{\lambda}_i$. For an eigenvector for the action of $H$ (the $e^{\lambda}_i$'s for instance) we call a {\em weight} its eigenvalue. One remarks from the expression of the action that the weights decrease $2$ by $2$ from $e^{\lambda}_0$ of weight $\lambda + r-1$, to $e^{\lambda}_{r-1}$ of weight $\lambda -r+1$ and so on, so that $\lambda$ is the {\em ``middle weight"} of $V_{\lambda}$.
 
Let $V_{\lambda}$ be the module of ``middle" weight $\lambda$ and of dimension $r$, for a $\lambda \in \C$ and $\lbrace e^{\lambda}_0,e^{\lambda}_1,\ldots,e^{\lambda}_{r-1} \rbrace$ its standard basis.

\begin{defn}\label{Rmatrixfortypical}
Let $\lambda$ and $\mu$ be elements of $\BC$. We define the morphism $\RRR$ from $V_{\lambda} \otimes V_{\mu}$ to $V_{\mu} \otimes V_{\lambda}$ as follows:
\[\RRR(\lambda,\mu) = c_{V_{\lambda},V_{\mu}}\].
\end{defn}

This operator $\RRR$ used as an R-matrix, provides the braid representations arising from the $\RT$-functor. %In the following sections, we focus on some special representations built from it.

\subsection{Construction of the representation for $\Mcg(0,4)$}\label{Construction}

We follow \cite{BCGP2} to give a basis of the vector space associated via the non semi-simple TQFT functor to the sphere with $4$ punctures. The definitions of typical module of $\Ubar$ can be found in Section \ref{Ubar} while tools as Clebsch-Gordan quantum coefficients and 6j-symbols are taken from \cite{C-M,CGP2}. From now on, we let $S_4$ be the sphere containing four marked points $p_1, p_2,p_3,p_4$. In order to compute the TQFT, we shall decorate the punctures using $\Ubar$ simple modules parametrized by complex numbers. Let $\lambda_1,\lambda_2,\lambda_3,\lambda_4$ be complex parameters in $\left( \C \backslash \Z \right) \bigcup r\Z$, considering $q$ to be a root of unity such that $q^{2r} = 1$. We recall from Section \ref{Ubar} that $\cat$ designates the category of $\Ubar$ weight modules. Let $S_4(\lambda_1,\lambda_2,\lambda_3,\lambda_4)$ be $S_4$ decorated by modules $V_{\lambda_i}\in \cat$ associated to each $p_i$, with $i=1,2,3,4$ (to fit with the decorated formalism \cite[Subsection~3.3]{BCGP2}). See Section \ref{Ubar} for the definition of the $\V_{\lambda_i}$'s. If $V=\left( (V_{\lambda_1},+),(V_{\lambda_2},+),(V_{\lambda_3},+),(V_{\lambda_4},+) \right) $, then $S_4(\lambda_1,\lambda_2,\lambda_3,\lambda_4)$ refers to the sphere with four punctures decorated by $V$, namely $S^2_V$ using notations from Section 6.1  of \cite{BCGP2}. We outline the construction.  

\begin{prop}[{\cite[Proposition~6.1]{BCGP2}}]\label{TQFTtoHom}
Let $\cV (\lambda_1,\lambda_2,\lambda_3,\lambda_4)$ be the $0$-graded vector space associated to $S_4(\lambda_1,\lambda_2,\lambda_3,\lambda_4)$ via the non semi-simple TQFT functor $\BV$ (from Theorem \ref{BCGPfunctor}). Algebraically, the space $\cV (\lambda_1,\lambda_2,\lambda_3,\lambda_4)$ is isomorphic to $\Hom _{\cat} (\BI,F(V))$ where $F(V) = V_{\lambda_1} \otimes V_{\lambda_2} \otimes V_{\lambda_3} \otimes V_{\lambda_4}$ is a module of $\cat$, and $\BI$ the identity object of $\cat$, namely the one dimensional $\Ubar$-module. 
\end{prop}
\begin{proof}[Idea of the proof]
The idea of this proposition is that the space $\cV (\lambda_1,\lambda_2,\lambda_3,\lambda_4)$ provided by the universal construction is generated by cobordisms between the empty set and $S_4(\lambda_1,\lambda_2,\lambda_3,\lambda_4)$ (Remark \ref{sketchofBHMV}). This corresponds to $\cat$-decorated ribbon tangles embedded inside the $3$-dimensional ball having $S_4(\lambda_1,\lambda_2,\lambda_3,\lambda_4)$ as boundary (tangles ending at punctures). Then the $\RT$-functor fully interprets these ribbon tangles as elements of $\Hom _{\cat} (\BI,F(V))$.
\end{proof}

%Proposition 6.1 of \cite{BCGP2} mentioned above gives an algebraic interpretation of the non semi-simple TQFT functor. Namely, it interprets cobordisms from the empty set to $S_4(\lambda_1,\lambda_2,\lambda_3,\lambda_4)$ in the category of decorated cobordisms with the $3$-dimensional ball as underlying manifold as elements of $\Hom _{\cat} (\BI,F(V))$ using the TQFT functor. Toward this picture, the cobordisms correspond to $\cat$-colored ribbon tangles embedded inside the $3$ dimensional ball and ending at punctures colored by the $V_{\lambda_i}$'s, so that the punctures come with arrows of the tangent space corresponding to the end of the ribbons. 

To transform a vector $v$ of $\cV (\lambda_1,\lambda_2,\lambda_3,\lambda_4)$ under the action of an element of the mapping class group one must just glue the corresponding mapping cylinder to the $S_4(\lambda_1,\lambda_2,\lambda_3,\lambda_4)$ at the extremity of the vector $v$ (interpreted as a cobordism) so to get a new cobordism between the empty set and $S_4(\lambda_1,\lambda_2,\lambda_3,\lambda_4)$ defined to be the image of $v$ under the mapping class action. The latter corresponds to the gluing of a ribbon sphere braid to the ribbon tangle corresponding to $v$. The ribbon aspect of the theory forces one to work with arrows instead of punctures $p_1,\ldots , p_4$, as extremities of ribbons are arrows. Then one must consider mapping classes fixing the arrows at the punctures, which correspond to mapping classes of the sphere with boundary components instead of punctures, but one can verify the following remark allowing us to deal with the whole mapping classes of the punctured sphere.

\begin{rmk}
A simple full Dehn twist around one puncture colored by a $\Ubar$ simple module gives a full twist to the arrow. In terms of morphism of the category $\cat$, this corresponds (through the $\RT$-functor) to a morphism from a simple module to itself. By Schur's Lemma, such morphism is diagonal.
\end{rmk}

Hence, a solution to avoid a restriction of the mapping class group is to consider the projective representations over the TQFT space $\cV (\lambda_1,\lambda_2,\lambda_3,\lambda_4)$ associated with $S_4(\lambda_1,\lambda_2,\lambda_3,\lambda_4)$, keeping simple punctures and forgetting the arrows at punctures. We will stick to this from now on and until the end of this section. In this framework of projective representations, it is not necessary to consider ribbons anymore, we simply consider $\cat$-colored tangles.  We introduce $\cat$-decorated trivalent graphs that will determine a basis of the TQFT later on.

\begin{defn}[$\CH$ and $\CI$ graphs]\label{HandIgraphs}
Let $\CH(\lambda_1,\lambda_2,\lambda_3,\lambda_4, \beta)$ be the decorated graph on the left of Figure \ref{Hspinegraph} embedded into the $3$-dimensional ball having the punctures $p_i$'s as endpoints in $S_4$ - the boundary of the ball. In this graph, $\beta$ is another complex parameter. A decoration $\lambda \in \BC$ refers to the module $V_{\lambda}$. In Figure \ref{Hspinegraph} the graph $\CI (\lambda_1,\lambda_2,\lambda_3,\lambda_4,\gamma)$ is also represented on the right, which corresponds to another vector of $\cV (\lambda_1,\lambda_2,\lambda_3,\lambda_4)$ used below. In what follows, we use a graph $\CG$ to refer to its image $\cV( \CG) \in \Hom _{\cat} (\BI,F(V))$ if no confusion arises in equations.

%\begin{figure}
%\raggedright
\begin{figure}[h]
\[\begin{array}{rl}\vcenter{\hbox{
\begin{tikzpicture}
\draw (0,0) circle (2);
\node[below left] (p1) at (225:2) {$p_1$};
\node[above left] (p2) at (135:2) {$p_2$};
\node[above right] (p3) at (45:2) {$p_3$};
\node[below right] (p4) at (-45:2) {$p_4$};
\coordinate (P1) at (225:2);
\coordinate (P2) at (135:2);
\coordinate (P3) at (45:2);
\coordinate (P4) at (-45:2);
\coordinate (B1) at (-0.5,0);
\coordinate (B2) at (0.5,0);
\draw[repmidarrow=50pt] (P1) -- (B1) node[midway,below right]{$\lambda_1$};
\draw[repmidarrow=50pt] (P2) -- (B1) node[midway,below left]{$\lambda_2$};
\draw[repmidarrow=50pt] (P3) -- (B2) node[midway,below right]{$\lambda_3$};
\draw[repmidarrow=50pt] (P4) -- (B2) node[midway,below left]{$\lambda_4$};
\draw[repmidarrow=50pt] (B1) -- (B2) node[midway,below]{$\beta$};
\end{tikzpicture}
}} &
\vcenter{\hbox{
\begin{tikzpicture}
\draw (0,0) circle (2);
\node[below left] (p1) at (225:2) {$p_1$};
\node[above left] (p2) at (135:2) {$p_2$};
\node[above right] (p3) at (45:2) {$p_3$};
\node[below right] (p4) at (-45:2) {$p_4$};
\coordinate (P1) at (225:2);
\coordinate (P2) at (135:2);
\coordinate (P3) at (45:2);
\coordinate (P4) at (-45:2);
\coordinate (A1) at (0,-0.5);
\coordinate (A2) at (0,0.5);
\draw[repmidarrow=50pt] (P1) -- (A1) node[midway,above left]{$\lambda_1$};
\draw[repmidarrow=50pt] (P2) -- (A2) node[midway,below left]{$\lambda_2$};
\draw[repmidarrow=50pt] (P3) -- (A2) node[midway,below right]{$\lambda_3$};
\draw[repmidarrow=50pt] (P4) -- (A1) node[midway,above right]{$\lambda_4$};
\draw[repmidarrow=50pt] (A1) -- (A2) node[midway,right]{$\gamma$};
\end{tikzpicture}
}} 
\end{array}
\]
\caption{Graphs $\CH(\lambda_1,\lambda_2,\lambda_3,\lambda_4, \beta)$ and $\CI (\lambda_1,\lambda_2,\lambda_3,\lambda_4,\gamma)$}
\label{Hspinegraph}
\end{figure}
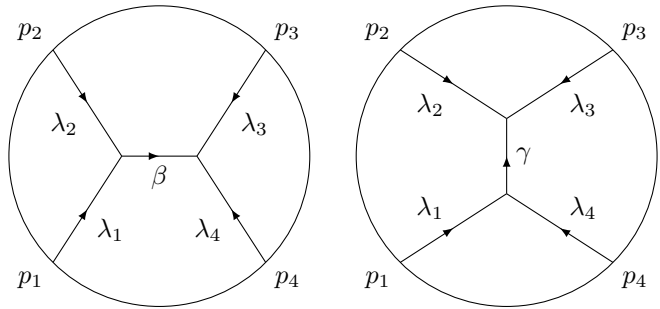

\end{defn}

\begin{rmk}\label{Hr}
As we only consider simple-module coloring of punctures, we can use the Clebsch-Gordan decomposition of tensor products of simple modules (see \cite[Section~1.3]{C-M}) to establish a correspondence between $\cat$-tangles and admissible trivalent graphs embedded in the ball, colored with elements of $\cat$ and ending at punctures. The word {\em admissible} refers here to the fact that the trivalent graphs must satisfy a relation at each node provided by the Clebsch-Gordan formula. Indeed, let $V_a$ and $V_b$ be two typical modules of middle weights $a$ and $b \in \left( \C \backslash \Z \right) \bigcup r\Z$. For $a,b \in \BC$ generic, it holds: $V_a \otimes V_b = \bigoplus_{a+b-c \in \Hr} V_c$ with $\Hr = \{ r-1, r-3, \ldots, -r+1 \}$, and any $\Ubar$ module map $V_c \to V_a \otimes V_b$ is a scalar multiple of the inclusion map of $V_c$ into $V_a \otimes V_b$ given in Theorem 1.7 of \cite{C-M}.

%The Clebsch-Gordan decomposition of the tensor product of $\Uq$-modules involves the direct sum of simple modules $V_c$ such that $a+b-c \in \Hr = \{ r-1, r-3, \ldots, -r+1 \}$ (namely $V_a \otimes V_b = \bigoplus_{a+b-c \in \Hr} V_c$), and so that any $\Ubar$ module map $V_c \to V_a \otimes V_b$ is a scalar multiple of the inclusion map of $V_c$ into $V_a \otimes V_b$ given in Theorem 1.7 of \cite{C-M}.
\end{rmk}

From Proposition \ref{TQFTtoHom} and from the construction of the TQFT functor $\BV$ from \cite{BCGP2} presented in Theorem \ref{BCGPfunctor}, one can check the following fact:

\begin{fact}[$\CH$ graphs basis]\label{Harebasis}
Let $ \cV(\lambda_1 , \lambda_2 , \lambda_3 , \lambda_4)$ be the $0$-graded TQFT space associated with $S_4(\lambda_1 , \ldots , \lambda_4)$ by functor $\BV$ (Theorem \ref{BCGPfunctor}). Then $\cV$ is isomorphic to the vector space generated by all $\cat$-decorated trivalent graphs inside the ball having ends at punctures, modulo the whole set of Relations (N a--j) of \cite[Section~2.2]{CGP2}. Moreover a basis of $\cV$ is given by the set of all graphs $\CH(\lambda_1,\lambda_2,\lambda_3,\lambda_4, \beta)$ satisfying the node condition (or  {\em admissible condition}), namely that the sum of parameters arriving to each vertex must be in $\Hr$.
\end{fact}
\begin{proof}[Idea of the proof]
The first step would be to interpret $\cat$-decorated tangles embedded in the ball (decorated cobordisms from the empty set to the sphere) as $\cat$-decorated trivalent graphs. This fact is an inherent tool in the quantum-module category $\cat$ and is a classical property of the non semi-simple $\RT$-functor. It works the same for quantum invariants of manifold from \cite{CGP2}. Namely, from the decomposition of tensor products of simple objects given by the Clebsch-Gordan formula (Remark \ref{Hr}), one obtains the formula of Proposition \ref{ClebschGordan} between graphs stated below. This gives a hint to pass from tangles to trivalent graphs. Once this step is done, one has to show that the family of $\CH$-graphs yields a basis of this space of graphs. 

One can check that Relations (N a--j) from \cite[Section~2.2]{CGP2} ensure that this family of graphs generates $\cV (\lambda_1,\lambda_2,\lambda_3,\lambda_4)$ (i.e. that any admissible trivalent graph can be expressed as a linear combination of the $\CH$ graphs using these relations). 

There is a little more work to get that the family is linearly independent. For instance, in the proof of \cite[Proposition~6.1]{BCGP2}, the pairing:
\[
\Hom _{\cat} (\BI,F(V)) \times \Hom _{\cat} (F(V),\BI) \to \BC
\]
is shown to be non-degenerate (as $F(V)$ is a projective $\Uq$-module). The pairing corresponds - in terms of cobordisms - to the one (schematically) presented in Remark \ref{sketchofBHMV}. One can compute it using $\CH$ graphs and their dual graphs corresponding to elements of $\Hom _{\cat} (F(V),\BI)$ and deduce the linear independence of these families. 

This type of proof is performed in \cite[Section~6.3]{BCGP2} to give a basis for the TQFT of empty surfaces. 
\end{proof}

\begin{rmk}[$\CI$ graphs basis]
The graphs $ \CI (\lambda_1, \lambda_2, \lambda_3, \lambda_4, \gamma_{\pm})$ correspond to another basis of $\cV (\lambda_1,\lambda_2,\lambda_3,\lambda_4)$, with the admissible values for $\gamma$ (see Remark \ref{HisI} below for instance).
\end{rmk}

%This proposition allows one to pass from decorated tangles to decorated trivalent graphs, and to represent elements of $\Hom _{\cat} (\BI,F(V))$ now as trivalent graphs embedded inside the ball and ending at punctures $p_1 , \ldots , p_4$. Thanks to the latter, we give a basis of $\cV (\lambda_1,\lambda_2,\lambda_3,\lambda_4)$ made of trivalent graphs colored with elements of $\cat$.

%\begin{figure}
%\begin{tikzpicture}
%\draw (0,0) circle (2);
%\node[below left] (p1) at (225:2) {$p_1$};
%\node[above left] (p2) at (135:2) {$p_2$};
%\node[above right] (p3) at (45:2) {$p_3$};
%\node[below right] (p4) at (-45:2) {$p_4$};
%\coordinate (P1) at (225:2);
%\coordinate (P2) at (135:2);
%\coordinate (P3) at (45:2);
%\coordinate (P4) at (-45:2);
%\coordinate (A1) at (0,-0.5);
%\coordinate (A2) at (0,0.5);
%\draw[repmidarrow=50pt] (P1) -- (A1) node[midway,above left]{$\lambda_1$};
%\draw[repmidarrow=50pt] (P2) -- (A2) node[midway,below left]{$\lambda_2$};
%\draw[repmidarrow=50pt] (P3) -- (A2) node[midway,below right]{$\lambda_3$};
%\draw[repmidarrow=50pt] (P4) -- (A1) node[midway,above right]{$\lambda_4$};
%\draw[repmidarrow=50pt] (A1) -- (A2) node[midway,right]{$\gamma$};
%\end{tikzpicture}
%\caption{$\CI (\lambda_1,\lambda_2,\lambda_3,\lambda_4,\gamma)$}
%\label{Ispinegraph}
%\end{figure}

In our case, the node conditions (``admissible conditions") are the following ones:
\begin{equation}
\label{M04basiscondition1}
\lambda_1 + \lambda_2 - \beta \in \Hr
\end{equation}
\begin{equation}
\label{M04basiscondition2}
\lambda_3 + \lambda_4 + \beta \in \Hr ,
\end{equation}

so a basis of $\cV (\lambda_1,\lambda_2,\lambda_3,\lambda_4)$ is given by the set $\{\CH(\lambda_1,\lambda_2,\lambda_3,\lambda_4, \beta)\}$ with all possible parameters $\beta$ such that Conditions \ref{M04basiscondition1} and \ref{M04basiscondition2} are satisfied.

\begin{Not}[$0$-graded, level $2$ TQFT]
We suppose from now on that $r=2$, then $\Hr=\{ -1, +1 \}$. Suppose also that $\lambda_4 = - \left( \lambda_1 + \lambda_2 + \lambda_3 \right)$, then we are left with three free parameters, namely $\lambda_1, \lambda_2, \lambda_3$. This set-up corresponds to the $0$-graded TQFT in the case $r=2$ (often referred to as the ``level $2$" non semi-simple TQFT), we denote the corresponding $0$-graded space $\cV(\lambda_1 , \ldots , \lambda_4)$.
\end{Not}

There are two possible graphs given by the two possible values for $\beta$. Let:
\[
\beta_+ = \lambda_1 + \lambda_2 + 1
\]
\[
\beta_- = \lambda_1 + \lambda_2 - 1 .
\]

We use the notations $\CH_+(\lambda_1, \lambda_2, \lambda_3, \lambda_4)$ and $\CH_-(\lambda_1, \lambda_2, \lambda_3, \lambda_4)$ to refer to the graphs $\CH(\lambda_1, \lambda_2, \lambda_3, \lambda_4,\beta)$ with $\beta=\beta_+=\lambda_1 + \lambda_2 +1$ and $\beta=\beta_-=\lambda_1 + \lambda_2 -1$ respectively; and $\CI_+(\lambda_1, \lambda_2, \lambda_3, \lambda_4)$ and $\CI_-(\lambda_1, \lambda_2, \lambda_3, \lambda_4)$ for the graphs $\CI(\lambda_1, \lambda_2, \lambda_3, \lambda_4,\gamma)$ with $\gamma=\gamma_+=\lambda_1 + \lambda_4 +1$ and $\gamma=\gamma_-=\lambda_1 + \lambda_4 -1$ respectively. We could have removed $\lambda_4$ from the arguments, as $\lambda_4$ is fixed, depending on the other arguments. This fact remains true even when we permute punctures, so that it permutes colors of the graphs, but the last argument will always be the opposite of the sum of the others, as $\lambda_i= - \sum_{j \neq i} \lambda_j$ for $i=1,2,3,4$. 

From Relations (N a--j) \cite[Section~2.2]{CGP2} and mentioned to define $\cV$, we will need three of them to build the representation that we recall in the three following propositions. 

\begin{prop}[{\cite[Equation~(N~i)]{CGP2}}]\label{ClebschGordan}

The following equality holds in $\cV(\lambda_1 , \ldots , \lambda_4)$:

\begin{align*}
\begin{array}{c}
\begin{tikzpicture}
\draw[repmidarrow=50pt] (0,0) -- (0,1.5) node[midway,left]{$a$};
\draw[repmidarrow=50pt] (0.7,0) -- (0.7,1.5) node[midway,right]{$b$};
\end{tikzpicture}
\end{array}
= \sum_{\gamma \in a+b+\Hr} \qd (\gamma)
\begin{array}{c}
\begin{tikzpicture}
\coordinate (P1) at (0.5,0.5);
\coordinate (P2) at (0.5,1);
\draw[repmidarrow=50pt] (0,0) -- (P1) node[midway,left]{$a$};
\draw[repmidarrow=50pt] (1,0) -- (P1) node[midway,right]{$b$};
\draw[repmidarrow=50pt] (P1) -- (P2) node[midway,left]{$\gamma$};
\draw[repmidarrow=50pt] (P2) -- (0,1.5) node[midway,left]{$a$};
\draw[repmidarrow=50pt] (P2) -- (1,1.5) node[midway,right]{$b$};
\end{tikzpicture}
\end{array}
\end{align*}

where graphs are considered to be the same outside the part of the picture drawn in the small ball considered here.

\end{prop}
 
\begin{prop}[{\cite[Equation~(N~j)]{CGP2}}]\label{HtoI}
There is the change of basis formula (between $\CH$ and $\CI$) that is obtained using what we call $6j$-symbols as follows:
\[
\CH_{\pm} (\lambda_1, \lambda_2, \lambda_3, \lambda_4) = \sum_{\epsilon = \pm 1} \qd(\lambda_1+\lambda_4 + \epsilon) \sjtop{\lambda_1}{\lambda_2}{\lambda_1+ \lambda_2 \pm 1}{\lambda_3}{-\lambda_4}{\lambda_1+\lambda_4 + \epsilon} \CI_{\epsilon} (\lambda_1, \lambda_2, \lambda_3, \lambda_4) .
\] 
where:
\begin{eqnarray*}
\sjtop{j_1}{j_2}{j_3}{j_4}{j_5}{j_6}=
(-1)^{r-1+B_{165}} \,
\dfrac{\{B_{345}\}! \, \{B_{123}\}!}
{\{B_{246}\}! \, \{B_{165}\}!} \,
\qbin{j_3+r-1}{A_{123}+1-r} \,
{\qbin{j_3+r-1}{B_{354}}}^{-1}\times
\\
\times\sum_{z = m}^{M}\,
(-1)^z\qbin{A_{165}+1}{j_5+z+r} \,
\qbin{B_{156}+z}{B_{156}} \,
\qbin{B_{264}+B_{345}-z}{B_{264}} \,
\qbin{B_{453}+z}{B_{462}}
\end{eqnarray*}
where $A_{xyz}=\frac{j_x+j_y+j_z+3(r-1)}{2}$,
$B_{xyz}=\frac{j_x+j_y-j_z+r-1}{2}$,
$m=\rm{max}(0,\frac{j_3+j_6-j_2-j_5}{2})$ and
$M=\rm{min}(B_{435},B_{165})$.
\end{prop}

\begin{prop}[{\cite[Equation~(N~g)]{CGP2}}]\label{braidCGP}
The following equality holds in $\cV(\lambda_1,\ldots,\lambda_4)$:
\begin{equation*}
\begin{array}{c} \hspace{-1.3mm}
        \raisebox{-4pt}{\begin{tikzpicture}[scale=0.8]
%\node (HG) at (0.8,0.2) {$\lambda$};
%\node (HD) at (2.2,0.2) {$\mu$};
\braid[style strands={1}{black}] s_1 ;
\draw (1,-1.5) arc (-180:-90:0.5) node[above] (S) { } arc (-90:0:0.5);
\draw[repmidarrow=50pt] (1,0) -- (1,0.2) node[near start,left]{$\lambda$};
\draw[repmidarrow=50pt] (2,0) -- (2,0.2) node[near start,left]{$\mu$};
\draw[repmidarrow=50pt] (1.5,-3) -- (S) node[midway,left]{$\beta$};
\end{tikzpicture}}
        \hspace{-1.9mm}\end{array}
=q^{\frac{-\lambda^2-\mu^2+\beta^2+(r-1)^2}{4}} \times
\begin{array}{c} \hspace{-1.3mm}
        \raisebox{-4pt}{\begin{tikzpicture}[scale=0.8]
\coordinate (HG) at (1,0.2);
\coordinate (HD) at (1.999,0.2);
%\node (BG) at (1,-1.9) {$\lambda$};
%\node (BD) at (2,-1.9) {$\mu$};
%\braid[style strands={1}{black}] s_1^{0};
\coordinate (BG) at (1,-1.5);
\coordinate (BD) at (1.999,-1.5);
\draw[repmidarrow=50pt] (BG)--(HG) node[midway,left]{$\mu$};
\draw[repmidarrow=50pt] (BD) -- (HD)node[midway,left]{$\lambda$};
\draw (BG) arc (-180:-90:0.5) node[above] (S) { } arc (-90:0:0.5);
\draw[repmidarrow=50pt] (1.5,-3) -- (S) node[midway,left]{$\beta$};
\end{tikzpicture}}
        \hspace{-1.9mm}\end{array}
%\caption{Relation from $\cat$.}
%\label{BraidMorphism.}
\end{equation*}
where we suppose that the graphs are the same everywhere outside the small ball drawn here. 
\end{prop}

\begin{rmk}\label{HisI}
We have the following equality, given by an obvious symmetry:
\begin{align*}
\begin{array}{c}\begin{tikzpicture}[scale=0.8]
\draw (0,0) circle (2);
\node[below left] (p1) at (225:2) {$p_1$};
\node[above left] (p2) at (135:2) {$p_2$};
\node[above right] (p3) at (45:2) {$p_3$};
\node[below right] (p4) at (-45:2) {$p_4$};
\coordinate (P1) at (225:2);
\coordinate (P2) at (135:2);
\coordinate (P3) at (45:2);
\coordinate (P4) at (-45:2);
\coordinate (B1) at (-0.5,0);
\coordinate (B2) at (0.5,0);
\draw[repmidarrow=50pt] (P1) -- (B1) node[midway,below right]{$\lambda_1$};
\draw[repmidarrow=50pt] (P2) -- (B1) node[midway,below left]{$\lambda_2$};
\draw[repmidarrow=50pt] (P3) -- (B2) node[midway,below right]{$\lambda_3$};
\draw[repmidarrow=50pt] (P4) -- (B2) node[midway,below left]{$\lambda_4$};
\draw[repmidarrow=50pt] (B1) -- (B2) node[midway,below]{$\beta$};
\end{tikzpicture} \end{array} = \begin{array}{c}\begin{tikzpicture}[scale=0.8]
\draw (0,0) circle (2);
\node[below left] (p1) at (225:2) {$p_1$};
\node[above left] (p2) at (135:2) {$p_4$};
\node[above right] (p3) at (45:2) {$p_3$};
\node[below right] (p4) at (-45:2) {$p_2$};
\coordinate (P1) at (225:2);
\coordinate (P2) at (135:2);
\coordinate (P3) at (45:2);
\coordinate (P4) at (-45:2);
\coordinate (A1) at (0,-0.5);
\coordinate (A2) at (0,0.5);
\draw[repmidarrow=50pt] (P1) -- (A1) node[midway,above left]{$\lambda_1$};
\draw[repmidarrow=50pt] (P2) -- (A2) node[midway,below left]{$\lambda_4$};
\draw[repmidarrow=50pt] (P3) -- (A2) node[midway,below right]{$\lambda_3$};
\draw[repmidarrow=50pt] (P4) -- (A1) node[midway,above right]{$\lambda_2$};
\draw[repmidarrow=50pt] (A1) -- (A2) node[midway,right]{$\beta$};
\end{tikzpicture}
\end{array}
\end{align*}
Using the notations introduced above, we get: $\CH(\lambda_1,\lambda_2,\lambda_3,\lambda_4,\beta)= \CI(\lambda_1,\lambda_4,\lambda_3,\lambda_4,\beta)$.
\end{rmk}

The last three properties allow us to compute the action in the $\CH$ basis. 
\begin{prop}
The action of standard generators $\sigma_1 , \sigma_2 , \sigma_3$ of $\Mcg(0,4)$ (Example \ref{4sphere}) over $\cV (\lambda_1 , \ldots , \lambda_4)$  in the $\CH$-basis are given by the following formulas:
\[\sigma_1 (\CH_{\pm}(\lambda_1,\lambda_2,\lambda_3,\lambda_4)) = q^{\frac{-\lambda_1^2-\lambda_2^2+\beta_{\pm}^2+(r-1)}{4}}\CH_{\pm}(\lambda_2,\lambda_1,\lambda_3,\lambda_4),\]
\[\sigma_3 (\CH_{\pm}(\lambda_1,\lambda_2,\lambda_3,\lambda_4)) = q^{\frac{-\lambda_3^2-\lambda_4^2+\beta_{\pm}^2+(r-1)}{4}} \CH_{\pm}(\lambda_1,\lambda_2,\lambda_3,\lambda_4)
\]
and:
\[
\sigma_2(\CH_{\pm}(\lambda_1, \lambda_2,\lambda_3,\lambda_4)) = f_{2,+}^{\pm}(\lambda_1,\lambda_3,\lambda_2,\lambda_4) \CH_+(\lambda_1,\lambda_3,\lambda_2,\lambda_4) + f_{2,-}^{\pm}(\lambda_1,\lambda_3,\lambda_2,\lambda_4) \CH_-(\lambda_1,\lambda_3,\lambda_2,\lambda_4)
\]
where:
\[ 
f_{2,+}^{\pm}(\lambda_1,\lambda_3,\lambda_2,\lambda_4) = \left( \sum_{\substack{\gamma = \lambda_1 +\\ \lambda_4 \pm 1}} \qd(b_+) \qd(\gamma) q^{\frac{-\lambda_2^2-\lambda_3^2+\gamma^2+(r-1)}{4}} \sjtop{\lambda_1}{\lambda_4}{\gamma}{\lambda_2}{-\lambda_3}{b_+} \sjtop{\lambda_1}{\lambda_2}{\beta_{\pm}}{\lambda_3}{-\lambda_4}{\gamma} \right)
\]
\[ f_{2,-}^{\pm}(\lambda_1,\lambda_3,\lambda_2,\lambda_4) = \left( \sum_{\substack{\gamma = \lambda_1 + \\ \lambda_4 \pm 1}} \qd(b_-) \qd(\gamma) q^{\frac{-\lambda_2^2-\lambda_3^2+\gamma^2+(r-1)}{4}} \sjtop{\lambda_1}{\lambda_4}{\gamma}{\lambda_2}{-\lambda_3}{b_-} \sjtop{\lambda_1}{\lambda_2}{\beta_{\pm}}{\lambda_3}{-\lambda_4}{\gamma} \right)
\]
\end{prop}
\begin{proof}
The idea to get the images of half-Dehn twists by the TQFT is to apply the twist to the graphs $\CH_{\pm}(\lambda_1,\lambda_2,\lambda_3,\lambda_4)$ corresponding to basis vectors. 

\begin{rmk}\label{permutedpuncturesofsphere}
Suppose that the twist involves a permutation $\tau$ of the punctures. Let $\tau(\lambda_1 , \lambda_2 , \lambda_3 , \lambda_4)$ be $(\lambda_{\tau(1)} , \ldots , \lambda_{\tau(4)})$ for $\tau \in \Sk_4$. We want to express the obtained graph in terms of the vectors $\CH_{\pm}(\tau(\lambda_1,\lambda_3,\lambda_2,\lambda_4))$ which yield a basis of the TQFT space $\cV(\tau(\lambda_1,\lambda_2,\lambda_3,\lambda_4))$ associated to the punctured sphere with permuted punctures.
\end{rmk}

This is done using the rules presented above. We use notations of Proposition \ref{presentationSphere} for the generators of $\Mcg(0,4)$. As $\sigma_1$ refers to the half-Dehn twist along $\left[ p_1, p_2 \right]$, we get that: 
\[\sigma_1 (\CH_{\pm}(\lambda_1,\lambda_2,\lambda_3,\lambda_4)) = q^{\frac{-\lambda_1^2-\lambda_2^2+\beta_{\pm}^2+(r-1)}{4}}\CH_{\pm}(\lambda_2,\lambda_1,\lambda_3,\lambda_4)\] with $\beta_{\pm}$ defined above.

In terms of graphs, the latter is illustrated below.

\begin{equation*}
\sigma_1(\CH(\lambda_1,\lambda_2,\lambda_3,\lambda_4,\beta)) =
\begin{array}{c}
\begin{tikzpicture}[scale=0.8]
\node[below left] (p1) at (225:2) {$p_2$};
\node[above left] (p2) at (135:2) {$p_1$};
\node[above right] (p3) at (45:2) {$p_3$};
\node[below right] (p4) at (-45:2) {$p_4$};
\coordinate (P1) at (225:2);
\coordinate (P2) at (135:2);
\coordinate (P3) at (45:2);
\coordinate (P4) at (-45:2);
\coordinate (B1) at (-0.5,0);
\coordinate (B2) at (0.5,0);
%\draw[repmidarrow=50pt] (P1) -- (B1) node[midway,below right]{$\lambda_1$};
%\draw[repmidarrow=50pt] (P2) -- (B1) node[midway,above right]{$\lambda_2$};
\draw[repmidarrow=50pt] (P3) -- (B2) node[midway,above left]{$\lambda_3$};
\draw[repmidarrow=50pt] (P4) -- (B2) node[midway,below left]{$\lambda_4$};
\draw[repmidarrow=50pt] (B1) -- (B2) node[midway,above]{$\beta$};
\draw[repmidarrow=50pt] (P1) -- (145:0.8) node[midway,left]{$\lambda_2$};
\draw[cross,repmidarrow=50pt] (P2) -- (215:0.8) node[midway,left]{$\lambda_1$};
\draw (145:0.8) -- (B1);
\draw (215:0.8) -- (B1);
\draw (0,0) circle (2);
\end{tikzpicture}
\end{array}
= q^{\frac{-\lambda_1^2-\lambda_2^2+\beta^2+(r-1)^2}{4}}
\begin{array}{c} \hspace{-1.3mm}
        \raisebox{-4pt}{\begin{tikzpicture}[scale=0.8]
\draw (0,0) circle (2);
\node[below left] (p1) at (225:2) {$p_2$};
\node[above left] (p2) at (135:2) {$p_1$};
\node[above right] (p3) at (45:2) {$p_3$};
\node[below right] (p4) at (-45:2) {$p_4$};
\coordinate (P1) at (225:2);
\coordinate (P2) at (135:2);
\coordinate (P3) at (45:2);
\coordinate (P4) at (-45:2);
\coordinate (B1) at (-0.5,0);
\coordinate (B2) at (0.5,0);
\draw[repmidarrow=50pt] (P1) -- (B1) node[midway,below right]{$\lambda_1$};
\draw[repmidarrow=50pt] (P2) -- (B1) node[midway,above right]{$\lambda_2$};
\draw[repmidarrow=50pt] (P3) -- (B2) node[midway,above left]{$\lambda_3$};
\draw[repmidarrow=50pt] (P4) -- (B2) node[midway,below left]{$\lambda_4$};
\draw[repmidarrow=50pt] (B1) -- (B2) node[midway,above]{$\beta$};
\end{tikzpicture}}
\hspace{-1.9mm}\end{array}
\end{equation*}
which is straightforward from Proposition \ref{braidCGP}. 

The same works for $\sigma_4$ so that one obtains:
\[\sigma_4 (\CH_{\pm}(\lambda_1,\lambda_2,\lambda_3,\lambda_4)) = q^{\frac{-\lambda_3^2-\lambda_4^2+\beta_{\pm}^2+(r-1)}{4}} \CH_{\pm}(\lambda_1,\lambda_2,\lambda_3,\lambda_4).
\]

To compute $\sigma_2$ which corresponds to the half Dehn-twist along $\left[ p_2 , p_3 \right]$, there is a little more work. The shortest way to express $\sigma_2(\CH_{\pm}(\lambda_1,\lambda_2,\lambda_3,\lambda_4))$ in terms of graphs $\CH_{\pm}(\lambda_1,\lambda_3,\lambda_2,\lambda_4)$ is to pass through the $\CI$ graphs as follows:

%\begin{equation}
%\begin{array}{rl}
\begin{align*}
\sigma_2(\CH(\lambda_1,\lambda_2,\lambda_3,\lambda_4,\beta)) & = \begin{array}{c}\begin{tikzpicture}[scale=0.8]
\node[below left] (p1) at (225:2) {$p_1$};
\node[above left] (p2) at (135:2) {$p_3$};
\node[above right] (p3) at (45:2) {$p_2$};
\node[below right] (p4) at (-45:2) {$p_4$};
\coordinate (P1) at (225:2);
\coordinate (P2) at (135:2);
\coordinate (P3) at (45:2);
\coordinate (P4) at (-45:2);
\coordinate (B1) at (-0.5,0);
\coordinate (B2) at (0.5,0);
\draw[repmidarrow=50pt] (P1) -- (B1) node[midway,below right]{$\lambda_1$};
\draw[repmidarrow=50pt] (P2) -- (B2) node[midway,below left]{$\lambda_3$};
\draw[cross,repmidarrow=50pt] (P3) -- (B1) node[midway,below right]{$\lambda_2$};
\draw[repmidarrow=50pt] (P4) -- (B2) node[midway,below left]{$\lambda_4$};
\draw[repmidarrow=50pt] (B1) -- (B2) node[midway,below]{$\beta$};
\draw (0,0) circle (2);
\end{tikzpicture} \end{array} \\
 & = \sum_{\gamma=\lambda_1+\lambda_4 \pm 1} \qd(\gamma) \sjtop{\lambda_1}{\lambda_2}{\beta}{\lambda_3}{-\lambda_4}{\gamma} \times \begin{array}{c}\begin{tikzpicture}[scale=0.8]
\node[below left] (p1) at (225:2) {$p_1$};
\node[above left] (p2) at (135:2) {$p_3$};
\node[above right] (p3) at (45:2) {$p_2$};
\node[below right] (p4) at (-45:2) {$p_4$};
\coordinate (P1) at (225:2);
\coordinate (P2) at (135:2);
\coordinate (P3) at (45:2);
\coordinate (P4) at (-45:2);
\coordinate (A1) at (0,-0.5);
\coordinate (A2) at (0,0.5);
\draw[repmidarrow=50pt] (P1) -- (A1) node[midway,above left]{$\lambda_1$};
%\draw[repmidarrow=50pt] (P2) -- (A2) node[midway,above right]{$\lambda_3$};
%\draw[repmidarrow=50pt] (P3) -- (A2) node[midway,above left]{$\lambda_2$};
\draw[repmidarrow=50pt] (P4) -- (A1) node[midway,above right]{$\lambda_4$};
\draw[repmidarrow=50pt] (A1) -- (A2) node[midway,right]{$\gamma$};
\draw[repmidarrow=50pt] (P2) -- (55:0.8) node[midway,below left]{$\lambda_3$};
\draw[cross,repmidarrow=50pt] (P3) -- (125:0.8) node[midway,below right]{$\lambda_2$};
\draw (A2) -- (125:0.8);
\draw (A2) -- (55:0.8);
\draw (0,0) circle (2);
\end{tikzpicture}
\end{array} \\
& = \sum_{\gamma=\lambda_1+\lambda_4 \pm 1} q^{\frac{-\lambda_2^2-\lambda_3^2+\gamma^2+(r-1)}{4}} \qd(\gamma) \sjtop{\lambda_1}{\lambda_2}{\beta}{\lambda_3}{-\lambda_4}{\gamma} \times \begin{array}{c}\begin{tikzpicture}[scale=0.8]
\node[below left] (p1) at (225:2) {$p_1$};
\node[above left] (p2) at (135:2) {$p_3$};
\node[above right] (p3) at (45:2) {$p_2$};
\node[below right] (p4) at (-45:2) {$p_4$};
\coordinate (P1) at (225:2);
\coordinate (P2) at (135:2);
\coordinate (P3) at (45:2);
\coordinate (P4) at (-45:2);
\coordinate (A1) at (0,-0.5);
\coordinate (A2) at (0,0.5);
\draw[repmidarrow=50pt] (P1) -- (A1) node[midway,above left]{$\lambda_1$};
%\draw[repmidarrow=50pt] (P2) -- (A2) node[midway,above right]{$\lambda_3$};
%\draw[repmidarrow=50pt] (P3) -- (A2) node[midway,above left]{$\lambda_2$};
\draw[repmidarrow=50pt] (P4) -- (A1) node[midway,above right]{$\lambda_4$};
\draw[repmidarrow=50pt] (A1) -- (A2) node[midway,right]{$\gamma$};
\draw[repmidarrow=50pt] (P2) -- (A2) node[midway,below left]{$\lambda_3$};
\draw[cross,repmidarrow=50pt] (P3) -- (A2) node[midway,below right]{$\lambda_2$};
\draw (0,0) circle (2);
\end{tikzpicture}
\end{array}
\end{align*}
%\end{array}
%\end{equation}

The second equality comes from Proposition \ref{ClebschGordan} while the last one from Proposition \ref{braidCGP}. The last graph must be expressed then back in terms of  $\CH_{\pm}(\lambda_1,\lambda_3,\lambda_2,\lambda_4)$ using Proposition \ref{HtoI}. Finally, we get the following expression for $\sigma_2 (\CH(\lambda_1,\lambda_2,\lambda_3,\lambda_4,\beta))$:
\begin{align*}
 \sum_{\substack{\gamma=\lambda_1+ \lambda_4 \pm 1}} q^{\frac{-\lambda_2^2-\lambda_3^2+\gamma^2+(r-1)}{4}} \qd(\gamma) \sjtop{\lambda_1}{\lambda_2}{\beta}{\lambda_3}{-\lambda_4}{\gamma} \left(\sum_{\substack{b=\lambda_1 \\ +\lambda_3 \pm 1}} \qd(b) \sjtop{\lambda_1}{\lambda_4}{\gamma}{\lambda_2}{-\lambda_3}{b} \begin{array}{c} \begin{tikzpicture}[scale=0.7]
\draw (0,0) circle (2);
\node[below left] (p1) at (225:2) {$p_1$};
\node[above left] (p2) at (135:2) {$p_3$};
\node[above right] (p3) at (45:2) {$p_2$};
\node[below right] (p4) at (-45:2) {$p_4$};
\coordinate (P1) at (225:2);
\coordinate (P2) at (135:2);
\coordinate (P3) at (45:2);
\coordinate (P4) at (-45:2);
\coordinate (B1) at (-0.5,0);
\coordinate (B2) at (0.5,0);
\draw[repmidarrow=50pt] (P1) -- (B1) node[midway,below right]{$\lambda_1$};
\draw[repmidarrow=50pt] (P2) -- (B1) node[midway,above right]{$\lambda_2$};
\draw[repmidarrow=50pt] (P3) -- (B2) node[midway,above left]{$\lambda_3$};
\draw[repmidarrow=50pt] (P4) -- (B2) node[midway,below left]{$\lambda_4$};
\draw[repmidarrow=50pt] (B1) -- (B2) node[midway,above]{$b$};
\end{tikzpicture}
\end{array} \right)
\end{align*}

We reorganize terms in order to get a more readable formula for the image of both vectors $\CH_+$ and $\CH_-$ expressed in the basis we were looking for:

\[
\begin{aligned}
\begin{array}{l}
\sigma_2(\CH_{\pm}(\lambda_1, \lambda_2,\lambda_3,\lambda_4)) = \\ 
= \left( \sum_{\substack{\gamma = \lambda_1 +\\ \lambda_4 \pm 1}} \qd(b_+) \qd(\gamma) q^{\frac{-\lambda_2^2-\lambda_3^2+\gamma^2+(r-1)}{4}} \sjtop{\lambda_1}{\lambda_4}{\gamma}{\lambda_2}{-\lambda_3}{b_+} \sjtop{\lambda_1}{\lambda_2}{\beta_{\pm}}{\lambda_3}{-\lambda_4}{\gamma} \right) \CH_+(\lambda_1,\lambda_3,\lambda_2,\lambda_4) \\

+ \left( \sum_{\substack{\gamma = \lambda_1 + \\ \lambda_4 \pm 1}} \qd(b_-) \qd(\gamma) q^{\frac{-\lambda_2^2-\lambda_3^2+\gamma^2+(r-1)}{4}} \sjtop{\lambda_1}{\lambda_4}{\gamma}{\lambda_2}{-\lambda_3}{b_-} \sjtop{\lambda_1}{\lambda_2}{\beta_{\pm}}{\lambda_3}{-\lambda_4}{\gamma} \right) \CH_-(\lambda_1,\lambda_3,\lambda_2,\lambda_4)\\

= f_{2,+}^{\pm}(\lambda_1,\lambda_3,\lambda_2,\lambda_4) \CH_+(\lambda_1,\lambda_3,\lambda_2,\lambda_4) + f_{2,-}^{\pm}(\lambda_1,\lambda_3,\lambda_2,\lambda_4) \CH_-(\lambda_1,\lambda_3,\lambda_2,\lambda_4).
\end{array}
\end{aligned}
\]

\end{proof}

Using these descriptions of the action of $\sigma_i$, for $i=1,2,3$, over $\cV (\lambda_1,\lambda_2,\lambda_3,\lambda_4)$ we associate to it an operator in $\PGL \left( \cV (\lambda_1,\lambda_2,\lambda_3,\lambda_4) ,\cV \left( \tau (\lambda_1,\lambda_2,\lambda_3,\lambda_4)\right)\right)$ with $\tau = \perm(\sigma_i) = (i,i+1) \in \fS_4$ permuting variables the way $\sigma_i$ permutes punctures. As we are not dealing with endomorphisms, we don't have a representation of the mapping class group, but a projective one over $\cV (\lambda_1,\lambda_2,\lambda_3,\lambda_4) \otimes \mathbb{C} \left[ \fS_4 \right] $. The latter is the induced representation of the pure mapping class group (consisting in the Torelli group made of mapping classes not permuting punctures), and uses a basis $\left\lbrace \CH_{\pm}(\tau(\lambda_1,\lambda_2,\lambda_3,\lambda_4)), \tau \in \fS_4 \right\rbrace$ consisting in graphs $\CH_{\pm}$ with colors permuted by permutations of $\fS_4$.% This definition of the induced representation is in the spirit of Definition \ref{inducedpure} for braids. 

\begin{defn}\label{quantumrepM04}
We define the following representation:
\[
\Phi: \left\lbrace 
\begin{array}{rcl}
\Mcg(0,4) & \to & \PGL (\cV (\lambda_1,\lambda_2,\lambda_3,\lambda_4)\otimes \mathbb{C} \left[ \fS_4 \right])\\
\sigma_i & \mapsto & \Phi(\sigma_i)
\end{array} \right.
\]
where:
\begin{align*}
\Phi(\sigma_i):\CH_{\pm}(\tau(\lambda_1,\lambda_2,\lambda_3,\lambda_4) \mapsto & f_{i,+}^{\pm}(\tau(\lambda_1,\lambda_3,\lambda_2,\lambda_4))\CH_+\left((i,i+1)\circ\tau \left(\lambda_1,\lambda_2,\lambda_3,\lambda_4\right) \right)\\
& +f_{i,-}^{\pm}(\tau(\lambda_1,\lambda_3,\lambda_2,\lambda_4))\CH_-\left((i,i+1)\circ\tau \left(\lambda_1,\lambda_2,\lambda_3,\lambda_4\right) \right).
\end{align*}
\end{defn}

\begin{rmk}[Normalization]\label{projectiveNormalization}

%(Remark \ref{quadraticframing}) 
As we are considering a projective action, we are going to normalize the representation canceling some factors. We will simplify quadratic terms interpreted as framing information by the $\RT$-functor and that we don't take into consideration in this work. All the quadratic terms in $\lambda_i$, for $i=1,2,3,4$, appear as factors of the operators. For instance in the expression of $\sigma_1(\CH(\lambda_1,\lambda_2,\lambda_3,\lambda_4,\beta))$, there is only $\beta^2$ depending on the basis vector, so that the associated operator has $q^{\frac{-\lambda_1^2-\lambda_2^2+(r-1)+(\lambda_1+\lambda_2)^2+1}{2}}$ as factor. For $\sigma_2$ we see that in the coefficients $q^{\frac{-\lambda_2^2-\lambda_3^2+\gamma^2+(r-1)}{2}}$, there is only $\gamma^2$ that varies with the basis vector. After developing both possible expressions for $\gamma$, we remark that  $q^{\frac{-\lambda_1^2-\lambda_4^2+(r-1)+(\lambda_1+\lambda_4)^2+1}{2}}$ factors the expression. For $\sigma_3$, we factorize by $q^{\frac{-\lambda_3^2-\lambda_4^2+(r-1)+(\lambda_3+\lambda_4)^2+1}{2}}$, so that we modify slightly the representation getting rid of these factor coefficients in the expression of matrices of the corresponding operators.

After the computation of the $6j$-symbols, we get matrices at level $r=2$, replacing $\lambda_4$ by $-(\lambda_1 + \lambda_2 + \lambda_3)$. We make the change of variables: $A_i= q^{2\lambda_i}$ for $i=1,2,3$ and we end up with the following expressions:
\begin{align}\label{matrixforM04quantum}
M_1(A_1,A_2,A_3)=\Matrice_{\cB_{(1)},\cB_{(1,2)}}\Phi(\sigma_1) & = \begin{pmatrix} \sqrt{A_1 A_2} & 0 \\ 0 & \frac{1}{\sqrt{A_1 A_2}} \end{pmatrix} \\\label{matrixforM04quantum2}
M_2(A_1,A_2,A_3)=\Matrice_{\cB_{(1)},\cB_{(2,3)}}\Phi(\sigma_2) & = (A_2^2 A_3^2-1)\begin{pmatrix}  \frac{-(1+A_3^2)}{A_1 A_2 A_3^2} & \frac{-(1+A_1^2)}{A_1 A_3} \\
\frac{A_1^2 A_2^2 A_3^2 + 1}{A_1 A_2^2 A_3} & \frac{(A_2^2+1)A_1}{A_2}
\end{pmatrix}
\end{align}
where $\cB_{\tau}$ designates the basis $\{ \CH_{\pm}(\tau(\lambda_1,\lambda_2,\lambda_3,\lambda_4))\}$ for $\tau \in \fS_4$, and $\Matrice_{B,B'}$ is the block corresponding to the image of $B$ in $B'$.

\end{rmk}

\subsection{Faithfulness}\label{faithfulness}

From now on, we restrict to the unicolored case, with $A_1 =A_2 = A_3 = A$, then $\Matrice_{\cB_{(1)},\cB_{(2,3)}}\Phi(\sigma_1)=\Matrice_{\cB_{(1)},\cB_{(3,4)}}\Phi(\sigma_3)$. 

We recall the exact sequence giving the homological representation of $\Mcg(0,4)$ presented in Example \ref{4sphere} together with the arrows associated with the representations involved here:

\begin{align}\label{AMUexact}
\begin{tikzpicture}
\tikzset{node distance=1.5cm, auto}
\node (A1) {1};
\node (A2) [right of=A1,,xshift=1.5cm] {$\BZ/2 \times \BZ/2$};
\node (A3) [right of=A2,,xshift=1.5cm] {$\Mcg(0,4) = G\ltimes N$};
\node (A4) [right of=A3,,xshift=1.9cm] {$\PSL(2,\BZ)$};
\node (A5) [right of=A4,,xshift=1.5cm] {$1$};
\node (B1) [below of=A1] { };
\node (B2) [right of=B1,,xshift=1.5cm] { };
\node (B3) [right of=B2,,xshift=1.5cm] {$\PGL(\cV)$};
\node (B4) [right of=B3,,xshift=1.5cm] { };
\node (B5) [right of=B4,,xshift=1.5cm] { };
%\node (C1) [below of=B1] {1};
%\node (C2) [right of=C1,,xshift=3cm] {$\pi_1 (UT(\hat{S}))$};
%\node (C3) [right of=C2,,xshift=3cm] {$\Mod(S)$};
%\node (C4) [right of=C3,,xshift=3cm] {$\Mod(S')$};
%\node (C5) [right of=C4,,xshift=3cm] {$1$};
%\node (D1) [below of=C1] {1};
%\node (D2) [right of=D1,,xshift=3cm] {$\pi_1 (\hat{S})$};
%\node (D3) [right of=D2,,xshift=3cm] {$\Mod(S,p)$};
%\node (D4) [right of=D3,,xshift=3cm] {$\Mod(S')$};
%\node (D5) [right of=D4,,xshift=3cm] {$1$};
%\node (E1) [below of=D1] { };
%\node (E2) [right of=E1,,xshift=3cm] {$1$};
%\node (E3) [right of=E2,,xshift=3cm] {$1$};
%\node (E4) [right of=E3,,xshift=3cm] { };
%\node (E5) [right of=E4,,xshift=3cm] { };
%%\node (D) [below of=A] {$1\otimes V_{n-1}$};
\draw[->] (A1) to node {  } (A2);
\draw[->] (A2) to node { } (A3);
\draw[->] (A3) to node { $\iota$ } (A4);
\draw[->] (A4) to node {  } (A5);
\draw[->] (A3) to node {$\Phi$} (B3);
%\draw[->] (B3) to node {  } (C3);
%\draw[->] (C3) to node {  } (D3);
%\draw[->] (D3) to node {  } (E3);
%\draw[->] (C1) to node { } (C2);
%\draw[->] (D1) to node { } (D2);
%\draw[->] (C2) to node { } (C3);
%\draw[->] (D2) to node { } (D3);
%\draw[->] (C4) to node { } (C5);
%\draw[->] (D4) to node { } (D5);
%\draw[->] (C3) to node { } (C4);
%\draw[->] (D3) to node { } (D4);
%\draw[->] (B2) to node { $\simeq$ } (B3);
%\draw[->] (C4) to node { $\simeq$ } (D4);
%\draw[->] (B2) to node { $\simeq$ } (B3);
%\draw[->] (B2) to node { $\simeq$ } (B3);
%\draw[->] (B2) to node { $\simeq$ } (B3);
%\draw[->] (C) to node {$\iota_{n,1}$} (B);
%\draw[->] (D) to node {$\iota_{n,1}$} (A);
%\draw[->] (D) to node {$\Gassner_{n-1}(\beta')$} (C);
\end{tikzpicture}  
\end{align}
where $G$ is the subgroup generated by $\sigma_1,\sigma_2$ and $N$ the one generated by $\alpha= \sigma_1 \sigma_3^{-1}$ and $\beta=\sigma_2 \sigma_1 \sigma_3^{-1} \sigma_2^{-1}$. The space $\cV$ designates the TQFT vector space associated with $S_4$ colored with $\lambda_1=\lambda_2=\lambda_3=\lambda$, last color still being the opposite of the sum of the others, tensored with permutations.

The following theorem concerns the restriction of this action to the group $G=\langle \sigma_1 , \sigma_2 \rangle$ which is isomorphic to $\PSL(2,\BZ)$.
%We make the important remark here, that in what follows, the action of $\sigma_3$ is not considered, so that $p_4$ will remain always at its initial place. As every other colors (of $p_1,p_2,p_3$) are the same, we will only need to look at the $2\times 2$ blocks $M_1$ and $M_2$. The fact that we don't consider $\sigma_3$ comes from the definition of $AMU$ on the generators of $\Mcg(0,4)$, providing that $\sigma_1 \sigma_3^{-1}$ is in the kernel:
%\[
%\iota(\sigma_1) = \iota(\sigma_3) = A= \begin{pmatrix} 1 & 1 \\ 0 & 1 \end{pmatrix} \text{ , } \text{ and } \iota(\sigma_2) = B = \begin{pmatrix} 1 & 0 \\ -1 & 1 \end{pmatrix}.
%\]

\begin{prop}\label{TQFTfaithfulonG}%\label{M(0,4)faithful}
The representation $\Phi_{|G}$ of $G$ provides a faithful representation of $\PSL(2,\BZ)$. 
\end{prop}
\begin{proof}
We recall the notations and framework. The morphism $\Phi$ is the quantum representation of $\Mcg (0,4)$ built on the sphere with $3$ marked points colored by $\lambda$, the last one by $-3\lambda$, with $\lambda$ a generic color of the category $\cat$, namely $\lambda \in \left( \BC \backslash \BZ \right) \bigcup r\BZ$ and $A=q^{-2\lambda}$. 

Let $b$ be a mapping class in $G$, we have that $\perm(b)$ is contained in the permutations that stabilized the last point, namely $p_4$ colored by $-3\lambda$. Then $\Phi(b) (\cV(\lambda, \lambda, \lambda, -3\lambda)\otimes \BC\left[ () \right]) \subset \left( \cV(\lambda, \lambda, \lambda, -3\lambda)\otimes \BC\left[ () \right] \right)$, if $()$ designates the identity permutation. This shows that for $\Phi_{|G}$ we can restrict ourselves to an action in $\PGL\left((\cV(\lambda,\lambda,\lambda,-3\lambda)\right)=\PGL\left(\cV\right)$ so that we still get a representation of $G$. We end the proof considering this representation.

We are going to work with the $s,t$ generators of $\PSL (2,\BZ)$, which are the images of $\sigma_1 \sigma_2$ and $\sigma_1 \sigma_2 \sigma_1$ respectively under the morphism $\iota$. Let $QS$ and $QT$ be their images under the quantum representation:
\[
\begin{aligned}
QS(A)= \Matrice_{\cV=\Vect(\CH_{\pm}(\lambda,\lambda,\lambda,-3\lambda))} (\Phi (\sigma_1 \sigma_2)) = \frac{1}{A^2-1} \begin{pmatrix} -1 & -A^2 \\ \frac{(A^2-1)^2}{A^2} + 1 & A^2  \end{pmatrix}
\\
QT(A)= \Matrice_{\cV=\Vect(\CH_{\pm}(\lambda,\lambda,\lambda,-3\lambda))} (\Phi (\sigma_1 \sigma_2 \sigma_1)) = \frac{1}{A^2-1} \begin{pmatrix} -A & -A \\ \frac{(A^2-1)^2}{A} + A & A \end{pmatrix}.
\end{aligned}
\]
These matrices are obtained from the appropriate products of matrices $M_i(A,A,A)$ with $i=1,2$ and after some renormalization making the determinant of $QS$ and $QT$ being equal to $1$, keep using the fact that we are considering projective matrices well defined up to multiplication by a scalar. We remark that they are well defined for $A\neq 0, \pm 1$. 

One can verify that $QS(A)^3 = QT(A)^2 = - \Id$ so that they are sent to the unit element of $\PSL(2,\BZ)$, this guarantees that we have a representation of $\PSL(2,\BZ)$. 
Let $P$ be the following matrix:
\[
P= \begin{pmatrix}
0 & 1 \\ \frac{A^2-1}{A} & -1
\end{pmatrix}
\] 
It is invertible for a generic choice of $A$ ($A \neq \pm 1 , 0$). 

Let's consider the representation $\Psi$ of $G$ obtained by conjugation of $\Phi$ by $P$, we have the following:
\[ \begin{aligned}
CS(A):= P^{-1} QS(A) P =  \begin{pmatrix} 0 & -1 \\ -1 & 0  \end{pmatrix} = S
\\
CT(A):= P^{-1} QT(A) P =  \begin{pmatrix} 0 & \frac{1}{A} \\ -A & 1  \end{pmatrix}
\end{aligned}
\]
For $A=1$ we get $CT(1)=T$ and we get that $\Psi$ is the standard representation of $\PSL(2,\BZ)$, which is faithful by definition. 

For $A=1$ we are not in the generic case, and $\Phi$ is not well defined (nor conjugated to $\Psi$).

Using the fact that the entries of matrices are meromorphic functions in the parameter $A$, and the density of possibilities for the choice of $A$ we get that these representations are generically faithful as follows.

Let $g\in \Mcg(0,4)$ and $G(A)$ its image under the above representation $\Psi$. As the entries are  analytic in $A$, let $L_g$ be the domain where $G(A) \neq \Id$. As $G(A) = \Id$ corresponds to zeros of analytic functions, it corresponds to isolated values of $A$ so that $L_g$ is dense in $\BC$. This is due to the fact that $\Psi$ is faithful for $A=1$ which guarantees that the functions considered are not zero everywhere and that its zeros are isolated. Hence, $\Psi$ is faithful for: $$A \in \bigcap_{g\in \Mcg(0,4)}L_g,$$  which is a countable intersection of dense spaces, hence dense by Baire's theorem. 
This proves that the representation $\Psi$ is generically faithful, and since it is generically conjugated to $\Phi$, $\Phi$ is also a generically faithful representation of $\PSL(2,\BZ)$. 
\end{proof}

\begin{thm}\label{M(0,4)faithful}
The projective representation $\Phi$ of  $\Mcg(0,4)$ is faithful.
\end{thm}
\begin{proof}
Let $h \in \Mcg(0,4)$. Suppose $h$ is in the kernel of $\Phi$. As $\Mcg(0,4) = G\ltimes N$, there exists a unique decomposition $h=g\cdot a$ with $g \in G$ and $a \in N$. For $h$ to be in the kernel of $\Phi$, $\perm(h)$ must be the identity permutation. This comes from the fact that only pure mapping classes are sent to block diagonal matrices. It implies $\perm(g) = \perm(a)^{-1}$. We've noticed at the beginning of the proof of Proposition \ref{TQFTfaithfulonG} that $\perm(g)$ fixes $p_4$, so that $a$ must also fix $p_4$. The element $a$ is one of the following: $\alpha, \beta , \alpha \beta$ of $1$ ($\alpha$ and $\beta$ are the generators of $N$ recalled above). One remarks that $\alpha$ sends $p_4$ in $p_3$, $\beta$ sends $p_4$ in $p_2$ and by $\alpha \beta$, $p_4$ is sent in $p_1$. This shows that the only possibility for $a$ is $1$ (i.e. the only element fixing $p_4$). Then $h = g \in G$ and can't be in the kernel of $\Phi$ from Proposition \ref{TQFTfaithfulonG}. 
\end{proof}

%\begin{tikzpicture}
%\draw (0,0) circle (2);
%\node[below left] (p1) at (225:2) {$p_2$};
%\node[above left] (p2) at (135:2) {$p_1$};
%\node[above right] (p3) at (45:2) {$p_3$};
%\node[below right] (p4) at (-45:2) {$p_4$};
%\coordinate (P1) at (225:2);
%\coordinate (P2) at (135:2);
%\coordinate (P3) at (45:2);
%\coordinate (P4) at (-45:2);
%\coordinate (B1) at (-0.5,0);
%\coordinate (B2) at (0.5,0);
%%\draw[repmidarrow=50pt] (P1) -- (B1) node[midway,below right]{$\lambda_1$};
%%\draw[repmidarrow=50pt] (P2) -- (B1) node[midway,above right]{$\lambda_2$};
%\draw[repmidarrow=50pt] (P3) -- (B2) node[midway,above left]{$\lambda_3$};
%\draw[repmidarrow=50pt] (P4) -- (B2) node[midway,below left]{$\lambda_4$};
%\draw[repmidarrow=50pt] (B1) -- (B2) node[midway,above]{$\beta$};
%\draw[repmidarrow=50pt] (P1) -- (145:0.8) node[midway,left]{$\lambda_2$};
%\draw[cross,repmidarrow=50pt] (P2) -- (215:0.8) node[midway,left]{$\lambda_1$};
%\draw (145:0.8) -- (B1);
%\draw (215:0.8) -- (B1);
%\end{tikzpicture}

\begin{coro}
Let $\Phi \in \Mcg(0,4)$ be a pseudo-Anosov mapping class. The stretching factor of $\Phi$ is detected by the TQFT representation $\BV$.
\end{coro}
\begin{proof}
From Theorem \ref{M(0,4)faithful}, the TQFT representation detects the representation of $\Mcg(0,4)$ in $\PSL(2,\BZ)$ and from Lemma 3.6 of \cite{AMU} the stretching factor is detected (as an eigenvalue) in the $\PSL(2,\BZ)$ representation of $\Mcg(0,4)$. 
\end{proof}

\section{TQFT representations of bigger mapping class groups}\label{TQFTbiggerspheres}

\subsection{The ADO set-up for braid representations}\label{typicalrep}

The category $\cat$ is a category of $\Uq$ modules for which the $\RT$ functor is shown to work, see \cite{BCGP2}. Namely let $\beta$ be a braid, and $\lambda_1 , \ldots , \lambda_n$ be complex parameters.

\begin{rmk}[ADO representations of the braid group]\label{ADOrep}
The restriction of the Reshetikhin-Turaev functor to braid associates a morphism of $\Ubar$-modules:
\[
\RT(\beta) \in \Hom_{\Ubar} \left( V_{\lambda_1} \otimes \cdots \otimes V_{\lambda_n} , V_{\perm(\beta)(\lambda_1)} \otimes \cdots \otimes V_{\perm(\beta)(\lambda_n)} \right)
\] 
as typical and atypical modules are objects of $\cat$. 
In particular, it provides a representation of $\PBn$ on $\End_{\Ubar}\left(V_{\lambda_1} \otimes \cdots \otimes V_{\lambda_n}\right)$. 
This representation is computable using the braiding defined in Definition \ref{Rmatrixfortypical}. 
\end{rmk}

\begin{defn}[ADO polynomials, \cite{ADO}]\label{ADOpolynomialnaive}
Let $\cat'$ be the sub-category of $\cat$ made of typical modules. A modified version of the Reshetikhin-Turaev functor restricted to $\cat'$-colored knots provides a knot invariants first introduced in \cite{ADO}, see \cite{ITO1} for another reference. This family of invariants is known as {\em colored Alexander polynomials} or {\em ADO polynomials}. 
\end{defn}

\subsection{Relations with punctured spheres}

Braid groups are mapping class groups of punctured disks. To relate them to mapping class groups of punctured spheres, we need an operation to pass from the disk to the sphere so that we introduce the {\em capping} operation. 
%\subsubsection{Capping the boundary}

Let $S$ be a surface with boundary, and $S'$ be the surface obtained from $S$ by closing one boundary component with a punctured disk. 
%In other words $S$ is included in $S'$ and $S'-S$ is a punctured disk. 
Set $p_0$ to be the puncture of the capping disk. Let $\beta$ be the loop in $S'$ corresponding to the boundary component of $S$. The group $\Mod (S, \{ p_1, \ldots , p_k \})$ is the subgroup of $\Mod (S)$ consisting of elements that fix the marked points $p_1,\ldots, p_k$, where $k\ge 0$, while $\Mod(S', \{p_0, p_1, \ldots , p_k\})$ is the subgroup of $\Mod(S')$ consisting of elements that fix the marked points $p_0,p_1,\ldots, p_k$. Then let $\Caping : \Mod(S, \{ p_1, \ldots , p_k\}) \to \Mod(S', \{p_0, p_1, \ldots , p_k\})$ be the induced homomorphism defined as follows: let $f$ be a homeomorphism of $S$ fixing $\partial S$ and representing a class of $\Mod(S, \{ p_1, \ldots , p_k\})$, and $\hat{f}$ be the homeomorphism of $S'$ which coincides with $f$ in $S$ and is the identity outside. Then $\Caping$ sends the class of $f$ to the one of $\hat{f}$ which turns out to be in $\Mod(S', \{p_0, p_1, \ldots , p_k \})$.
% describing the morphism $\Caping$, and describing the capping of a boundary component.

\begin{proposition}[{\cite[4.2.5]{F-M}}]\label{cappingboundarysequence}
The morphism $\Caping$ satisfies the following exact sequence:
\[
1 \to \langle \tau_{\beta} \rangle \to \Mod(S, \{ p_1, \ldots , p_k \}) \xrightarrow{ \Caping } \Mod(S', \{p_0, p_1, \ldots , p_k \}) \to 1
\]
where $\tau_{\beta}$ refers to the Dehn twist along $\beta$ and the first injection is the inclusion.
\end{proposition}

We use this operation to relate representations of braid groups to quantum representation of mapping class group of punctured spheres. 
First we relate quantum (pure) braid representations with non semi-simple TQFT representations of the (pure) mapping class groups of punctured spheres. We restrict the study to {\em pure} groups for an easier reading, but everything can be generalized to whole groups by considering the {\em induced representations}.% (see Definition \ref{inducedpure}). 

Let $q$ be a $2r^{th}$ root of unity, $\lambda_1 , \ldots , \lambda_n \in (\C\setminus \Z)\cup r\Z$ be complex colors, $V_{\lambda_1} , \ldots V_{\lambda_n}$ be the associated typical modules of $\cat$, and set $V = V_{\lambda_1} \otimes \cdots \otimes V_{\lambda_n}$. Let $\beta \in \PBn$ be a braid and $\RT(\beta) \in \End( V)$ its ADO-type representation introduced in Remark \ref{ADOrep} defined using the functor $\RT$. Let $\cV$ be the ($0$-graded) TQFT vector space associated by $\BV$ with the sphere with $n+1$ punctures, $p_1,\ldots , p_{n+1}$ and such that: $p_1$ is decorated by $V_{\lambda_1}$ and so on until $p_n$ is decorated by $V_{\lambda_n}$ and $p_{n+1}$ is decorated by $V^*$ the dual space of $V$. We recall the capping morphism from Proposition \ref{cappingboundarysequence} in the case of the disk with $n$ marked points:
\[
\Caping : \PBn \to \PMcg(0,n+1)
\]
where $\PBn$ is the pure braid group on $n$ strands, and $\PMcg(0,n)$ is the pure mapping class group of the sphere with $(n+1)$ punctures. Namely, the latter corresponds to the Torelli group of the punctured sphere and is made of mapping classes that leave the punctures fixed pointwise. The morphism $T = \BV \circ \Caping$ restricted to $\cV$ (the $0$-graded sub-space) provides a representation of $\PBn$ ($\BV$ is the TQFT functor). 

We recall from \cite[Proposition~6.1]{BCGP2} that the ($0$-graded) TQFT space $\cV$ is isomorphic to $\Hom_{\cat} ( \BI , V \otimes V^*)$. 
%Let $X \in V^*$, then the action of the pure braid group defined by $T$ over $\Hom_{\cat} ( \BI , V \otimes V^*)$ restricts to $\Hom_{\cat} ( \BI , V \otimes X)$.
%\end{fact}
%\begin{proof}

%\end{proof}
Let $\phi$ be the following injective morphism:
\[
\phi : \bapp
\End(V) & \to & \End\left( \Hom_{\cat} ( \BI , V \otimes V^*) \right) \\
M & \to & M \otimes \Id
\eapp
\]

\begin{lemma}\label{commutationRTandTQFT}
The following diagram is commutative:
\[
\begin{tikzcd}[column sep=small]
\PBn \arrow{r}{\RT}  \arrow{rd}{T} 
  & \End(V) \arrow{d}{\phi} \\
    & \End \left( \Hom_{\cat} ( \BI , V \otimes  V^*) \right) 
\end{tikzcd}
\]
\end{lemma}
\begin{proof}
For a braid $\beta \in \PBn$ we must show that:
\[
T(\beta) = \RT(\beta) \otimes \Id_{V^*}.
\]
By composition by $\Caping$, $\PBn$ acts over the $(n+1)$ punctured sphere by mapping classes fixing $p_{n+1}$. The mapping cylinder associated with an element of $\PBn$ is the identity cobordism in a small disk containing $p_{n+1}$ as the only puncture. By gluing this cylinder to a cobordism generating the TQFT interpreted as an element of $\Hom_{\cat} ( \BI , V \otimes V^*)$, it is easy to see that the morphism is the identity over the $\Uq$-module decorating $p_{n+1}$, namely $V^*$, and on the the $n$ other punctures it is by construction obtained by applying the $\RT$-functor. It proves the lemma.
\end{proof}

\begin{prop}\label{typicalincludedTQFT}
If there exists a $2r^{th}$ root of unity $q$ and colors $\lambda_1 , \ldots , \lambda_n \in (\C\setminus \Z)\cup r\Z$ such that the representation $\RT$ of $\PBn$ is faithful then the representation of $\PMcg(0,n+1)$ (suitably decorated as above) provided by $\BV$ is faithful. 
\end{prop}
\begin{proof}
The proof is a direct consequence of Lemma \ref{commutationRTandTQFT}, of the surjectivity of $\Caping$ and of the injectivity of $\phi$. Namely, as for any $\beta \in \PBn$:
\[
\BV \circ \Caping (\beta) = \RT(\beta) \otimes \Id_{V^*}
\]
(previous lemma) and that any element in $\Mcg(0,4)$ can be written $\Caping (\beta)$ for some $\beta$, if $\RT$ is faithful so is $\BV$. 
\end{proof}

The latter shows that the question of the faithfulness of punctured sphere mapping class group TQFT representations is included in the following open question.

\begin{oquestion}\label{ADOfaithfulness}
Are the ADO representations of braid groups introduced in Remark \ref{ADOrep} faithful? 
\end{oquestion}

From \cite{Jules1} for instance, it is known that quantum representations are recovered by Lawrence's representations \cite{Law} (in \cite{Jules1} the case of $q$ being a root of unity is treated). In \cite{Big1} and \cite{Kra}, it is proved that the second level of Lawrence representations are generically faithful, and from \cite{Z2} that the family of Lawrence representations are faithful in general (except for the first level). The word {\em generically} stands for a generic set of parameters $(q,\lambda_1 , \ldots , \lambda_n) \in \BC^{n+1}$. For instance the faithfulness proof of \cite{Big1} relies on the {\em key lemma}. This lemma uses extensively the Laurent polynomial structure of coefficients.%, and a study of some coefficient of the noodle-fork pairing defined to be the maximal coefficient with respect to some lexical order on monomials. 
Its proof crashes down whenever one wants to specialize the proof for $q$ being a root of unity. In this sense, the quantum representations are ``generically faithful'' but the question whether they are faithful at roots of unity (ADO set-up) is still open, so is the question of the faithfulness of TQFT representations of the punctured spheres. 

The case of the torus was studied in \cite{BCGP2} and led to an analog of Theorem \ref{M(0,4)faithful}. 

\begin{theorem}[{\cite[Theorem~6.28]{BCGP2}}]
The non semi-simple TQFT projective representation of the mapping class group of the torus, provided by the functor $\BV$, is faithful modulo its center. 
\end{theorem}

These ``small" cases (in terms of genus) are first steps for an answer to the following general question. 

\begin{question}[{\cite[Question~1.7.(1)]{BCGP2}}]
Let $\Sigma$ be a surface. Is the non semi simple TQFT representation of $\Mod( \Sigma )$ over $\BV(\Sigma)$ faithful?
\end{question}

%\section{MISSING IN THIS CHAPTER.}
%
%Dans cette partie, il manque:
%\begin{itemize}
%%\item La toute fin de la preuve de M(0,4) fidèle: détecter le groupe fini (facile: regarder Birman). 
%%\item Fidélité de TQFT(M(0,n)) inclue dans fidélité des représentations de tresses à couleurs complexes (besoin d'aide Francois).
%\item Dire que colored BKL inclue dans quantum: soit dire que c'est inclus dans le dernier chapitre, soit le faire à la main...
%\item Eventuellement donner l'idée de colorer un $(n+1)$-ème brin par $1$ et faire une récurrence avec la suite de Birman pour tenter d'aborder la fidélite. Mais surtout, dire que ca a été fait pour Gassner et que ça a donné un noyau très raffiné: cf article de KNUDSON
%%\item TQFT du tore à un trou ou M(0,6)? Pas le temps...
%\item open question: les TQFT non semi-simples calculent le stretching factor?
%\item open question: interpréter toutes les strates homologiques? Fait dans le chapitre suivant!
%\end{itemize}

\newpage
\thispagestyle{empty}


\begin{thebibliography}{99999}

\baselineskip15pt

 
\bibitem[ADO]{ADO} 
Y. Akutsu, T. Deguchi, and T. Ohtsuki {\em Invariants of colored links}, J. Knot Theory Ramifications, (1992), no. 2, 161--184.

\bibitem[AMU]{AMU}
J.E. Andersen, G. Masbaum, and K.Ueno, {\em  Topological Quantum Field Theory and the Nielsen-Thurston classification of M(0,4)},  Mathematical Proceedings of the Cambridge Philosophical Society {\bf 141} (2006), 477-488.

%\bibitem[An]{An}
%C. Anghel, {\em A Homological model for the coloured Jones polynomials}, arXiv:1712.04873,(2017). 
%
%\bibitem[B-B]{BaBo}
%G. Band, P. Boyland {\em The Burau estimate for the entropy of a braid}, Algebraic \& Geometric Topology {\bf 7}, (2007), 1345--1378.
%
%\bibitem[B-N]{B-N}
%D.Bar-Natan, {\em  A Note on the Unitarity Property of the Gassner Invariant},  Bulletin of Chelyabinsk State University (Mathematics, Mechanics, Informatics) {\bf 3-358-17} (2015), 22--25.
%
%\bibitem[Bas]{Bas}
%S. Baseilhac, {\em Quantum coadjoint action and the 6j-symbols of Uq sl(2)}, in Interactions between Hyperbolic Geometry, Quantum Topology and Number Theory, AMS Cont. Math. Proc. Ser. {\bf 541}, (2011) 103--144. 
%
\bibitem[Big]{Big1}
%S. Bigelow, {\em Braid groups are linear}, J. Amer. Math. Soc. {\bf 14} (2000), 471--486.
%
%\bibitem[Big1]{Big0}
%S. Bigelow, {\em Homological representations of the Iwahori-Hecke algebra}, Geom. Topol. Monographs {\bf 7}, (2004), 493--507.
%
%\bibitem[Big2]{Big2}
%S. Bigelow, {\em The Lawrence–Krammer representation}, Proc. Sympos. Pure Math. {\bf 71}, Amer.Math. Soc. Providence, (2003), 51--68.
%
%\bibitem[Big3]{Big3}
%S. Bigelow, {\em A homological definition of the Jones polynomial}, Invariants of knots and 3-manifolds (Kyoto, 2001), Geom. Topol. Monogr., volume {\bf 4} (2002), 29--41.
%
%\bibitem[Big4]{Big4}
%S. Bigelow, {\em A homological definition of the HOMFLY polynomial}, Algebraic and Geometric Topology, (2007). 

\bibitem[Bir]{Bir}
J. Birman, {\em Braids, Links and Mapping Class Groups.}, Annals of Mathematics Studies {\bf 82} Princeton University Press, Princeton, New Jersey, (1975).

\bibitem[BCGP]{BCGP2}
C. Blanchet, F. Costantino, N. Geer, B. Patureau, {\em  Non semi-simple TQFTs, Reidemeister torsion and Kashaev's invariants}, Advances in Mathematics {\bf 301} (2016), 1--78.

\bibitem[BHMV]{BHMV}
C. Blanchet, N. Habegger, G. Masbaum, and P. Vogel, {\em  Three-manifold invariants derived from the Kauffman bracket},  Topology {\bf 31} (1992), no. 4, 685--699.

%\bibitem[Br]{Brown}
%R. Brown, {\em Topology and Groupoids}, Booksurge, (2006). 
%
%
%\bibitem[RFB]{Nielsentheory}
%R.F. Brown, {\em More about Nielsen Theories and Their Applications}, Handbook of Topological Fixed Point Theory, (2005), Springer, Dordrecht, 433--462.
%
%
%\bibitem[BFGT]{BFGT}
%P. Bushlanov, B. Feigin, A. Gainutdinov, I. Tipunin, {\em Lusztig limit of quantum $\slt$ at root of unity and fusion of $(1, p)$ Virasoro logarithmic minimal models}, Nuclear Physics, B {\bf 818} (3), 179--195.

\bibitem[C-P]{C-P}
V. Chari, A. N. Pressley, {\em A guide to Quantum Groups}, Cambridge University Press, (1995).

\bibitem[CGP1]{CGP1}
F. Costantino, N. Geer, B. Patureau, {\em  Some remarks on the unrolled quantum group of $\slt$},  J. Pure Appl. Algebra {\bf 219}, no. 8 (2015), 3238--3262. 

\bibitem[CGP2]{CGP2}
F. Costantino, N. Geer, B. Patureau, {\em  Quantum invariants of $3$-manifolds via link surgery presentations and non-semi-simple categories},  Journal of Topology  {\bf 7} (2014), 1005--1053.

\bibitem[CGP3]{CGP3}
F.Costantino, N.Geer, B.Patureau, {\em Relations between Witten-Reshetikhin-Turaev and non semi-simple $\slt$ $3$-manifold invariants},  Algebraic Geometry and Topology,   {\bf 15} (2015), 1363--1386.

%\bibitem[CGP4]{CGP4}
%F.Costantino, N.Geer, B.Patureau, {\em Quantum invariants of 3-manifolds via link surgery presentations and non-semi-simple categories}, Journal of Topology (2014), {\bf 7} (4), 1005--1053. 

\bibitem[C-M]{C-M}
F. Costantino, J. Murakami, {\em On SL(2,C) quantum 6j-symbol and its relation to the hyperbolic volume}, Quantum Topology, {\bf 4} (2013), no.3, 303--351.

%\bibitem[C-J]{CJ}
%E. Creath, D. Jakelic, {\em Highest-weight vectors in tensor products of Verma modules for $\Uq$}, Boletim da Sociedade Paranaense de Matemática, Vol {\bf 36}, No 4, (2018).
%
%
%\bibitem[DCP]{DCP}
%C. De Concini, C. Procesi {\em Quantum Groups}, in ``D-modules, Representation Theory and Quantum Groups (Venice 1992)”, Springer Verlag, Lect. Notes Math. 1565 (1993) 31--140. 

\bibitem[DR]{Marco}
M. De Renzi, {\em Non-Semisimple Extended Topological Quantum Field Theories} To Appear in Memoirs of the American Mathematical Society. 

%\bibitem[Drin]{Drin}
%V. G. Drinfel’d, {\em Quasi-Hopf algebras}, Leningrad Math. {\bf J.1} (1990), 1419--1457.
%
%\bibitem[FH]{FH}
%E. Fadell, S. Husseini, {\em The Nielsen number on surface}, Topological Methods in Nonlinear Functional Analysis, Contemp. Math. {\bf 21}, (1983), 59--98.

\bibitem[F-M]{F-M}
B. Farb, D. Margalit, {\em A Primer on Mapping Class Groups}, Princeton University Press (2012).

%\bibitem[F-W]{FW}
%G. Felder, C. Wieczerkowski, {\em Topological Representations of the Quantum Group $\Uq$}, Commun. Math. Phys. {\bf 138} (1991), 583--605.
%
%\bibitem[GG1]{GG}
%D. Goncalvez, J. Guaschi, {\em Fixed points of $n$-valued maps on surfaces and the Wecken property - a configuration space approach}, Science China Mathematics {\bf 60}, (2017), 1561--1574.
%
%\bibitem[GG2]{GG2}
%D. Goncalvez, J. Guaschi, {\em Fixed points of n-valued maps, the fixed point property and the case of surfaces - a braid approach}, Indag. Math. {\bf 29}, (2018), 91--124.
%
%\bibitem[Gor]{Gor}
%L. Gorniewicz, {\em On the Lefschetz Fixed Point Theorem}, Handbook of Topological Fixed Point Theory, Springer, Dordrecht, 43--82.
%
%\bibitem[Hab]{Hab}
%K. Habiro, {\em An integral form of the quantized enveloping algebra of sl2 and its completions}, J. Pure Appl. Alg. {\bf 211} (2007), 265--292.
%
%\bibitem[Ha]{Ha}
%E. Hart, {\em Algebraic Techniques for Calculating the Nielsen Number on Hyperbolic Surfaces}, Handbook of Topological Fixed Point Theory, (2005), Springer, Dordrecht 463--487.
%
%\bibitem[Hat]{Hat}
%A. Hatcher, {\em Algebraic Topology}, Cambridge University Press, (2002). 
%
%\bibitem[H-J]{HJ}
%H.-H. Huang, B. Jiang, {\em Braids and periodic solutions}, Topological Fixed Point Theory and Applications, Lecture Notes in Math. {\bf 1411}, Springer-Verlag, Berlin, Heidelberg, New York, (1989), 107--123.
%
%\bibitem[Hu]{Hu}
%S. Husseini, {\em Generalized Lefschetz numbers}, Trans. Amer. Math. {\bf 272}, (1982), 247--274.
%
%\bibitem[H-L]{L-H}
%V. Huynh and T. T. Q. L\^e, {\em On the Colored Jones Polynomial and the Kashaev invariant}, Fundam. Prikl. Mat.  {\bf 11}  (2005),  no. 5, 57--78.
%
%\bibitem[Ik]{Ik}
%A. Ikeda, {\em Homological and Monodromy Representations of Framed Braid Groups}, Communications in Mathematical Physics (2018), {\bf 359}, 1091--1121.
%
%\bibitem[Ito]{Itogarside}
%T. Ito, {\em Reading the dual Garside length of braids from homological and quantum representations}, Comm. Math. Phys. {\em 335} (2015) 345--367.
%
\bibitem[Ito1]{ITO1}
T. Ito, {\em  A homological representation formula of colored Alexander invariant},  Adv. Math.  {\bf 289} (2016), 142--160.
%
%\bibitem[Ito2]{Ito2}
%T. Ito, {\em Topological formula of the loop expansion of the colored Jones polynomials}, Trans. Amer. Math. Soc., (2019). 
%
%\bibitem[J-K]{JK}
%C. Jackson and T. Kerler, {\em The  Lawrence-Krammer-Bigelow  representations  of  the  braid groups via $\Uq$}, Adv. Math, {\bf 228}, (2011), 1689--1717.
%
%\bibitem[J]{J}
%B. Jiang, Lectures on Nielsen Fixed Point Theory, Contemp. Math., vol. 14, Amer. Math Soc., Providence, Rhode Island, (1983).
%
%
%\bibitem[J2]{J2}
%B. Jiang, {\em A primer of Nielsen fixed point theory}, Handbook of Topological Fixed Point Theory, Springer, Dordrecht, 617--645.
%
%\bibitem[J-W]{JW}
%B. Jiang, S. Wang {\em Twisted Topological Invariants Associated with Representations} Topics in Knot Theory.,  (1993), {\bf 399} Springer, Dordrecht, 211--227.

\bibitem[M]{Jules1}
J. Martel, {\em A homological model for $\Uq$ Verma-modules and their braid representations}, arXiv:2002.08785, math.GT.

\bibitem[Kas]{Kas}
C. Kassel, {\em Quantum Groups.}, Graduate Texts in Mathematics {\bf 155} Springer-Verlag, New York, (1995).

\bibitem[K-T]{K-T}
C. Kassel, V. Turaev, {\em Braid Groups}, Springer (2008).

%\bibitem[Kar]{Kar}
%K. Karvounis {\em Braid group action on projective quantum $\slt$ modules},	arXiv:1909.11162 [math.GT]. 
%
%
%\bibitem[Knu]{Knu}
%K. Knudson, {\em On the kernel of the Gassner representation}, Archiv der Mathematik {\bf 85}, (2005), 108--117. 
%
%
%\bibitem[K0]{K0}
%T. Kohno, {\em Monodromy representations of braid groups and Yang-Baxter equations}, Ann. Inst.Fourier {\bf 37} (1987), 139--160.
%
%\bibitem[K1]{Kohmulti}
%T. Kohno, {\em Hyperplane arrangements, local system homology and iterated integrals},Arrangements of Hyperplanes—Sapporo 2009, 157--174, Mathematical Society of Japan, Tokyo, Japan, 2012. 
%
%\bibitem[K2]{Koh}
%T. Kohno, {\em Quantum and homological representations of braid groups}, Configuration Spaces - Geometry, Combinatorics and Topology, Edizioni della Normale (2012), 355--372.
%
\bibitem[Kra]{Kra}
D. Krammer, {\em Braid groups are linear}, Ann. Math. {\bf 155} (2002), 131--156.
%
\bibitem[Kra1]{Kra1}
D.Krammer, {\em  Braid groups are linear},  Annals of Mathematics {\bf 155} (2002), 451--486.
%
\bibitem[Law]{Law}
R. Lawrence, {\em Homological representations of the Hecke algebra}, Comm. Math. Phys. {\bf 135} (1990), 141--191.
%
%\bibitem[Len]{Len}
%S. Lentner, {\em The unrolled quantum group inside Lusztig's quantum group of divided powers}, Preprint (2017), arXiv:1702.05164.
%
%\bibitem[Lus]{Lus}
%G. Lusztig, {\em Finite dimensional Hopf algebras arising from quantum groups} , J. Am. Mat. Soc. {\bf 3} (1990) , 257--296.
%
%\bibitem[Lus2]{Lus2} 
%G. Lusztig, {\em Quantum groups at roots of $1$}, Geom. Ded., {\bf 35} (1990), 89--114.
%
%\bibitem[Mac]{Mack}
%K. C. H. Mackenzie, {\em General theory of Lie groupoids and Lie algebroids}, London Mathematical Society Lecture Notes, {\bf 213}, (2005).
%
%
%\bibitem[Mat]{Mat}
%T. Matsuoka, {\em Periodic Points and Braid Theory}, Handbook of Topological Fixed Point Theory, (2005), Springer, Dordrecht, 171--216.
%
%\bibitem[Mu-Mu]{MuMu}
%H. Murakami, J. Murakami, {\em The colored Jones polynomials and the simplicial volume of a knot}, Acta Math., {\bf 186}, (2001), 85--104.
%
%\bibitem[NO-W]{NOW}
%B. Norton-Odenthal, P. Wong {\em A relative generalized Lefschetz number}, Topology and its Applications {\bf 56(2)}, (1994), 141--157.

\bibitem[Oht]{Oht}
T. Ohtsuki, {\em Quantum invariants. A study of knots, 3–manifolds, and their sets}, Series on Knots and Everything {\bf 29} World Scientific Publishing Co., Inc., River Edge, NJ, (2002).

%\bibitem[P-P]{P-P}
%L. Paoluzzi, L. Paris, {\em A  note  on  the  Lawrence-Krammer-Bigelow representation}, Algebr. Geom. Topol., {\bf 2} (2002), 499--518. 

\bibitem[R-T]{RT}
N.  Reshetikhin, V.G.  Turaev, {\em Ribbon  graphs  and  their  invariants  derived  from  quantum  groups}, Comm. Math. Phys., {\bf 127} (1990), no. 1, 1--26.

\bibitem[RT2]{RT2}
N.Yu. Reshetikhin, V.G. Turaev, {\em Invariants of 3-manifolds via link polynomials and quantum groups}, Invent. Math. {\bf 103}  (1991), 547--597.

%\bibitem[Sal]{Sal}
%M. Salvetti, {\em Topology of the complement of real hyperplanes in $\mathbb{C}^N$}, Invent. Math. {\bf 88}, (1987),  no. 3, 603--618.
%
%\bibitem[S-V]{SchVar}
%V. Schechtman and A. Varchenko, {\em Quantum Groups and Homology of Local Systems}, ICM-90 Satellite Conference Proceedings, Algebraic Geometry and Analytic Geometry, 182--197. 
%
%
%\bibitem[Sch]{Sch}
%H. Schirmer, {\em A Survey of Relative Nielsen Fixed Point Theory}, Proceedings of the Conference on Nielsen Theory and Dynamical Systems, Contemp. Math. {\bf 152}, (1993), 291--309. 
%
%
%\bibitem[Z1]{Z1}
%H. Zheng, {\em A  reflexive  representation  of  braid  groups}, J. Knot Theory Ramifications, {\bf 14} (2005), 467--477.
%  
\bibitem[Z2]{Z2}
H. Zheng, {\em Faithfulness of the Lawrence representation of braid groups}, arXiv:math/0509074v1 .




\end{thebibliography}
\end{document}